\documentclass[11pt]{amsart}
\usepackage[a4paper, margin=2cm]{geometry}
\usepackage[leqno]{amsmath}
\usepackage{amssymb}
\usepackage{enumerate}
\usepackage{xcolor}
\usepackage{subfigure}
\usepackage{morefloats}
\usepackage{graphicx}
\usepackage{hyperref}
\usepackage{mathrsfs}
\usepackage[normalem]{ulem}

\numberwithin{equation}{section}

\newcommand{\eps}{\varepsilon}

\newcommand{\D}{\mathbf{D}}
\newcommand{\E}{\mathbf{E}}

\newcommand{\h}{\mathbf{H}}
\newcommand{\N}{\mathbf{N}}
\newcommand{\Z}{\mathbf{Z}}

\newcommand{\p}{\mathbf{P}}

\newcommand{\R}{\mathbf{R}}

\newcommand{\CA}{\mathcal {A}}

\newcommand{\CC}{\mathcal {C}}
\newcommand{\CD}{\mathcal {D}}

\newcommand{\CF}{\mathcal {F}}

\newcommand{\CL}{\mathcal {L}}

\newcommand{\CT}{\mathcal {T}}

\newcommand{\CW}{\mathcal W}
\newcommand{\CWW}{\mathcal V}

\newcommand{\CH}{\mathcal {H}}

\newcommand{\SLE}{{\rm SLE}}

\newcommand{\CLE}{{\rm CLE}}

\newcommand{\BCLE}{\rm BCLE}

\newcommand{\wt}{\widetilde}

\newcommand{\giv}{\,|\,}

\renewcommand{\d}{\delta}

\newcommand{\sol}[1]{{}}

\newcommand{\wh}{\widehat}

\definecolor{pinegreen}{rgb}{0.0, 0.47, 0.44}

\newcommand{\IG}{{\mathrm{IG}}}

\numberwithin{equation}{section}

\newtheorem{theo}{Theorem}[section]
\newtheorem{rema}[theo]{Remark}
\newtheorem{lemm}[theo]{Lemma}
\newtheorem{prop}[theo]{Proposition}
\newtheorem{coro}[theo]{Corollary}
\newtheorem{state}[theo]{Statement}

\begin{document}

\title[$\CLE_{\kappa'}$ on LQG]{Non-simple conformal loop ensembles on Liouville quantum gravity and the law of CLE percolation interfaces}

\author{Jason Miller}
\address{Statslab, Center for Mathematical Sciences, University of Cambridge, Wilberforce Road, Cambridge CB3 0WB, UK}
\email {jpmiller@statslab.cam.ac.uk}

\author{Scott Sheffield}
\address{
Department of Mathematics, MIT,
77 Massachusetts Avenue,
Cambridge, MA 02139, USA}
\email{sheffield@math.mit.edu}

\author{Wendelin Werner}
\address{Department of Mathematics, ETH Z\"urich, R\"amistr. 101, 8092 Z\"urich, Switzerland} 
\email{wendelin.werner@math.ethz.ch}

\dedicatory{This paper is dedicated to the memory of Harry Kesten}

\begin{abstract}
We study the structure of the Liouville quantum gravity (LQG) surfaces that are cut out as one explores a conformal loop-ensemble CLE$_{\kappa'}$ for $\kappa'$ in $(4,8)$ that is drawn on an independent $\gamma$-LQG surface for $\gamma^2=16/\kappa'$. The results are similar in flavor to the ones from our paper \cite {CLE_LQG} dealing with CLE$_{\kappa}$ for $\kappa$ in $(8/3,4)$, where the loops of the CLE are disjoint and simple. In particular, we encode the combined structure of the LQG surface and the CLE$_{\kappa'}$ in terms of stable growth-fragmentation trees or their variants, which also appear in the asymptotic study of peeling processes on decorated planar maps.

  This has consequences for questions that do a priori not involve LQG surfaces: Our previous paper \cite {cle_percolations} described the law of interfaces obtained when coloring the loops of a CLE$_{\kappa'}$ independently into two colors with respective probabilities $p$ and $1-p$. This description was complete up to one missing parameter $\rho$. The results of the present paper about CLE on LQG allow us to determine its value in terms of $p$ and $\kappa'$.
  
  It shows in particular that CLE$_{\kappa'}$ and CLE$_{16/\kappa'}$ are related via a continuum analog of the Edwards-Sokal coupling between FK$_q$ percolation and the $q$-state Potts model (which makes sense even for non-integer $q$ between $1$ and $4$) if and only if $q=4\cos^2(4\pi / \kappa')$. This provides further evidence for the long-standing belief that CLE$_{\kappa'}$ and CLE$_{16/\kappa'}$ represent the scaling limits of FK$_q$ percolation and the $q$-Potts model when $q$ and $\kappa'$ are related in this way. 
  
  Another consequence of the formula for $\rho(p,\kappa')$ is the value of half-plane arm exponents for such divide-and-color models (a.k.a. fuzzy Potts models) that turn out to take a somewhat different form than the usual critical exponents for two-dimensional models.
  
  \end {abstract}

\maketitle

\section{Introduction} 

Most of this paper will be devoted to the study of the collection of quantum surfaces that one obtains when drawing a non-simple conformal loop ensemble ($\CLE$) on top of an independent Liouville quantum gravity (LQG) surface.  This study will imply statements for $\CLE$ that do not involve LQG and that we choose to briefly present in the first two sections of this introduction. 

\subsection {A divide-and-color exponent} 
\label {S11}
In view of the fact that this paper is dedicated to the memory of Harry Kesten, it seems fitting to start it with one very particular sub-instance of the results that will be derived here which have direct consequences for a lattice-based model that is directly related to Bernoulli percolation on the square grid. 

Start with critical Bernoulli bond percolation on $\Z \times \N$ (i.e., where edges are open or closed independently  with probability $1/2$). This defines a configuration on edges, which in turn partitions the vertices into clusters.
Next, we choose a parameter $p \in (0,1)$, and we color the clusters independently in red and blue with probability $p$ and $1-p$.  
We are now interested in the event $E_R = E_R (p)$ that there exists a path of red sites joining the origin to
the semi-circle of radius $R$ around the origin. The results of the present paper will essentially imply that:

\begin {state}
\label {statement1}
If critical Bernoulli percolation is conformally invariant in the scaling limit, then $\p[ E_R (p)] = R^{-a(p) + o(1)}$ as $R \to \infty$, where 
$$a(p) = \frac {1}{6} \left( 2  - \frac 3 \pi  \arctan \frac { p\sqrt { 3} }{2-p}    \right) \left( 1 - \frac 3 \pi 
\arctan \frac { p\sqrt { 3}}{ 2-p}\right).$$ 
\end {state}

We can make the following four comments at this early stage: 
\begin {enumerate} 
 \item The appearance of the $\arctan (\cdot)$ function in the above formula suggests that its derivation will involve arguments somewhat different from those used to derive the ``usual'' critical exponents, which are computed only using SLE martingales. Indeed, as we shall explain in the present paper, these $\arctan$ type formulas appear to be a by-product of the decomposition of LQG-type surfaces in terms of L\'evy trees. 
 \item This formula is part of a bigger picture. Similar statements hold when critical Bernoulli percolation is replaced by a critical FK$_q$ random cluster model for $q \in (0,4)$ (so that we are now dealing with models known as fuzzy Potts models 
\cite {mv1995fuzzypotts,hagg1999fuzzypotts}). For instance, for the FK$_2$-Ising model, the formula for this one-arm half-plane exponent is  
\begin {equation} 
a(p) = \frac 1 6 \left(1-\frac 2 \pi \arctan \frac {p}{1-p}\right) \left(3 - \frac 4 \pi \arctan \frac {p}{1-p}\right). 
\end {equation}
 As we shall explain in the next section, these formulas follow from the detailed description of the red/blue interfaces obtained when one colors loops in a Conformal Loop Ensemble independently. 
 \item A similar statement can be formulated when one starts with site percolation on the triangular lattice (and then colors the ``cluster of edges'' independently) instead of Bernoulli percolation on the square lattice. This has the advantage that the result is then unconditional (as this percolation model is known to be conformally invariant) but it is then a little less natural (similarly, for the aforementioned Ising-FK$_2$ model, the result is also unconditional). We will not discuss the discrete to continuum convergence here and leave it for some upcoming paper in which these discrete divide-and-color exponents will be discussed further.
 \item The exponent $1/3$ that shows up in the $p \to 0+$ limit for $a(p)$ in Statement \ref {statement1} is the usual one-arm boundary exponent for critical percolation (which is what one would expect when one looks at the limit of the exponents) -- the same remark applies for the exponent $1/2$ that appears in the limit for the FK$_2$ model.
\end {enumerate}

\subsection{The $q(\kappa')$ formula for CLE percolation} 

CLEs are random families of loops in a simply connected domain (one can for instance consider the unit disk $\D$) that satisfy certain natural properties (conformal invariance and some version of a spatial Markov property) that make them the natural candidates for scaling limits of interfaces in critical two-dimensional models from statistical physics with a second-order phase transition \cite{SHE_CLE,SHE_WER_CLE}. They are parameterized by a real parameter $\kappa \in (8/3, 8)$; each loop in a $\CLE_\kappa$ is a loop-version of the Schramm-Loewner evolution ($\SLE_\kappa$) \cite{S0} and one can divide the CLEs into two regimes. In the regime $\kappa \in (8/3,4]$, a $\CLE_\kappa$ consists of a pairwise disjoint collection of simple loops which do not intersect the domain boundary while in the regime $\kappa \in (4,8)$ a $\CLE_\kappa$ consists of a collection of non-simple loops which can touch each other and the boundary of the domain. It is now customary (and we will use this notation throughout the present paper) to denote the non-simple CLEs by $\CLE_{\kappa'}$ for $\kappa' \in (4,8)$ (this  will prevent some confusion when discussing the ``duality'' statements).  The collection of loops are defined directly in the continuum with no reference to discrete models, but it can be useful to have in mind some of the main conjectures relating discrete models to CLE, as this can help to guide our intuition and our understanding of CLEs.

The following conjectures are particularly relevant to the results that we will present in this section (conjectures related to random planar maps will be behind the scenes when we will discuss CLE on LQG):

\begin{itemize}
\item The collection of interfaces (which are all loops on a medial lattice) in a critical FK$_q$-percolation model with free boundary conditions converges to a $\CLE_{\kappa'}$. When $q$ increases from $0^+$ to $4$, the corresponding value of 
$\kappa'$ should decrease from $8^-$ to $4$. This convergence has been proved in the cases $q=0^+$ (the uniform spanning tree) \cite{LSW04}, $q=1$ \cite{S01,CN06} (but only for site percolation on the triangular lattice, which is not really a planar bond percolation model) and $q=2$ \cite{s2010ising,ks2016fkconvergence,gw2019fk} (which is  the case related to the Ising model). Note that by duality, the same is essentially true when one considers an FK$_q$ model with wired boundary conditions.
\item  Consider a critical $q$-Potts model with uniform boundary conditions (say with color $1$ on the boundary) for $q=2,3,4$, and consider the scaling limit of the law of the cluster of color $1$ that contains the boundary points. Its inner boundary consists of a collection of closed disjoint loops. When the mesh of the lattice goes to $0$, this collection of loops should converge to the outermost loops in a $\CLE_\kappa$ for some value $\kappa \in (8/3, 4]$ (the limit of the boundary-touching cluster would be the set of points surrounded by no $\CLE_\kappa$ loop, that is called the $\CLE_{\kappa}$ carpet).
\end{itemize}

In the discrete setting, for integer $q \ge 2$, the Potts model and the FK$_q$ percolation models can be coupled as follows (see for instance \cite{GRIM_CLUSTER_BOOK}): When one chooses one of the $q$ colors uniformly at random and independently for each of the FK$_q$ clusters, then one obtains the Potts model (which was the initial motivation to study FK-percolation in \cite {FK}), and conversely, the FK$_q$ percolation can be viewed as ``Bernoulli-bond percolation'' in each of the Potts clusters (i.e., edges joining two sites with different colors are closed, and one tosses an independent biased coin for each of the other ones) --- this is sometimes referred to as the Edwards-Sokal coupling after \cite {ES}.  

This {\em suggests} that the conformal loop ensembles should have the following properties, which can be roughly stated without reference to any discrete model (and without requiring $q$ to be an integer).  Suppose that $\kappa \in (8/3, 4)$ and $\kappa' = 16 / \kappa \in (4,6)$.

\begin{enumerate}[(a)]
\item A $\CLE_{\kappa'}$ can be viewed as a model for critical Bernoulli percolation {\em within} a $\CLE_\kappa$ carpet.
\item Conversely, suppose we start with a $\CLE_{\kappa'}$ (in its nested version) and consider its collection of clusters (corresponding to wired boundary conditions --- we will detail how to define them in the next paragraph).  Then we color in blue each of its clusters independently with some probability $1/q(\kappa')$, except that the outermost cluster (which contains the boundary of the domain) is colored blue regardless. Then {\em the blue connected component touching the boundary is distributed like a $\CLE_\kappa$ carpet.}
\end{enumerate}

Let us briefly explain how to define the collection of clusters that are defined by a $\CLE_{\kappa'}$. We say that a point is surrounded by a $\CLE_{\kappa'}$ loop if the index of the loop around the point is non-zero -- we call $i(\CL)$ the set of such points. By convention, we will also view the domain boundary $\partial \D$ as one of the $\CLE_{\kappa'}$ loops.  We say that a loop in the $\CLE_{\kappa'}$ is an $n$th level loop if it is surrounded by exactly $n$ other loops in the $\CLE_{\kappa'}$ (so $\partial \D$ is the only $0$th level loop). 
 For each $n$th level loop $\CL$ in the $\CLE_{\kappa'}$ such that $n \ge 0$ is even, we define the cluster $K(\CL)$ surrounded by $\CL$ to be the closure of $i(\CL) \setminus \cup_{\CL'} i(\CL')$, where the union of $\CL'$ is taken over all $(n+1)$st generation loops surrounded by $\CL$.  The boundary cluster $K (\partial \D)$ in this $\CLE_{\kappa'}$ will be the closure of set of points $z$ with the property that the index of any of the $\CLE_{\kappa'}$ loop around $z$ is $0$ --- this is sometimes called the $\CLE_{\kappa'}$ gasket, by analogy with the Sierpinski gasket. (This definition of cluster mimics the definition of clusters associated to an FK-model with wired boundary conditions.)

The two results (a) and (b) were actually established in \cite{cle_percolations}, building on various inputs and in particular on the imaginary geometry couplings \cite{MS_IMAG} with the Gaussian free field (GFF), but except for the case $\kappa'=16/3$, the value $q(\kappa')$ was not determined. One outcome of the present paper is the following statement: 

\begin {theo}[The $q(\kappa')$ formula for CLE percolation]
\label {thm1}
The value $q(\kappa')$ in (b) is equal to $ 4 \cos^2 (4 \pi / \kappa')$.
\end {theo}

This therefore provides a direct derivation of the relation between $q$ and $\kappa'$, without reference to any discrete model calculation and completes the solution to \cite[Problem~8.10]{SHE_CLE}. Note that as explained in \cite{cle_percolations}, symmetry reasons implied that $q(16/3) = 2$ (which is consistent with the fact that the FK$_2$ model that is related to the Ising model converges to $\CLE_{16/3}$ while the Ising model converges to $\CLE_3$ \cite{BEN_HONG}), but this was the only value of $\kappa'$ for which $q(\kappa')$ was known. So for instance, the fact that  $q(\kappa')= 3$ for $\kappa' = 24/5$ (so that $\CLE_{10/3}$ carpets should describe the scaling limits of $(q=3)$-Potts clusters) is new. In particular, establishing the convergence of FK$_3$ to CLE$_{24/5}$ would then automatically imply the joint convergence of FK$_3$ and the coupled $3$-state Potts model to CLE$_{24/5}$ with the coupled CLE$_{10/3}$ obtained by the coloring procedure.  We see also that $q=4$ is the maximal possible value, just out of continuum CLE considerations (the threshold at $q=4$ features also in the nature of the phase transition for planar FK$_q$ models derived in \cite{DST,Detal}).  

To our knowledge, the only other instance where the conjectural relation between $q$ and $\kappa'$ has been derived from the continuum objects, is via CLE crossing events as described in \cite{mw2018connection}. For a brief survey of other arguments that led to this conjecture, we refer to \cite{mw2018connection}. 

It is actually shown in \cite{cle_percolations} that for any $\kappa' \in (4,8)$, if one uses any value $p$ (for $p <1$) to color the $\CLE_{\kappa'}$ clusters using the same procedure as before ($p$ replaces $1/q$), then one obtains one of the so-called boundary conformal loop ensembles $\BCLE_\kappa (\rho)$ (where $\rho$ is a parameter which is determined by $p$) for $\kappa = 16/ \kappa'$. The present paper will show {\em which} $\BCLE_\kappa(\rho)$ is obtained for each value of $p$. One consequence of this fact is that it allows us to describe the ``non-blue'' clusters that appear in the holes of the blue clusters in the above construction (another way of phrasing this result deals with the so-called 
``full'' $\SLE_{\kappa'}^\beta (\kappa' -6)$ processes as defined in \cite{cle_percolations} and more precisely with the identification of the law of their trunk).

\begin{figure}[pt!]
\includegraphics[width=0.48\textwidth]{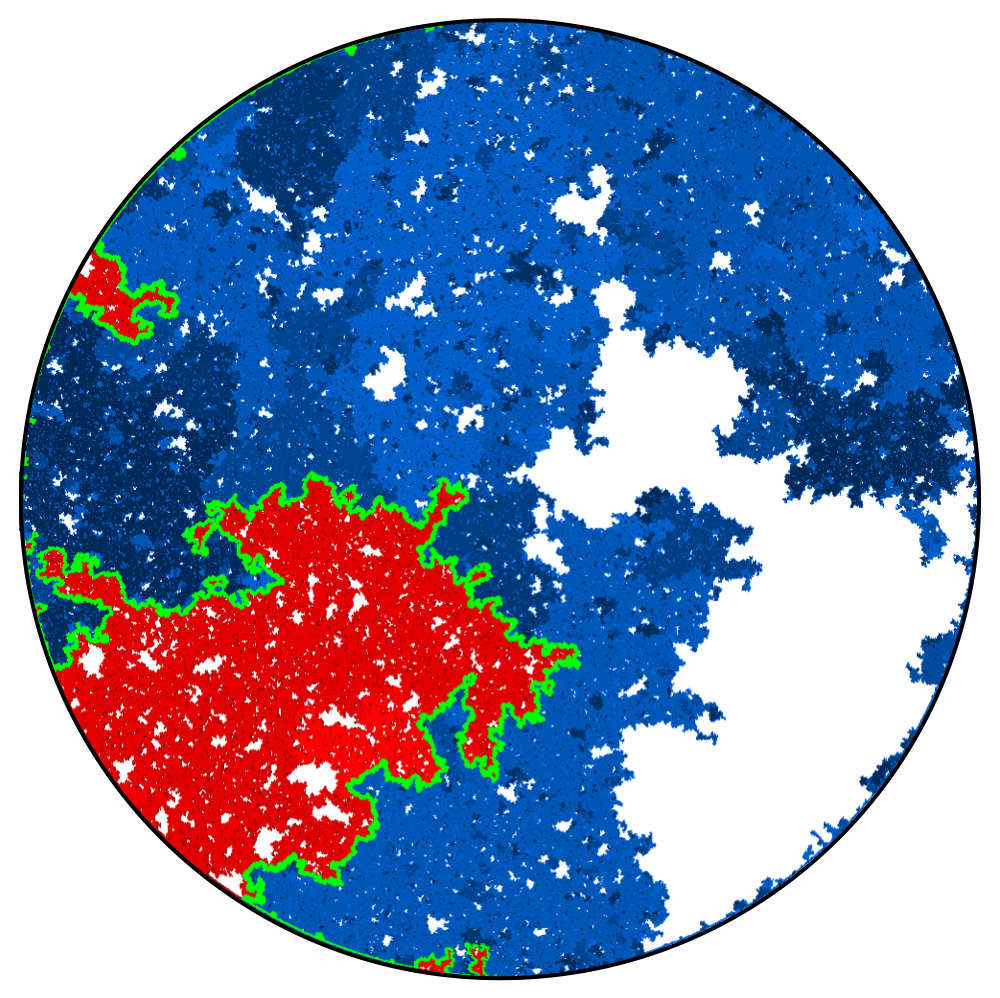}\hspace{0.02\textwidth}\includegraphics[width=0.48\textwidth]{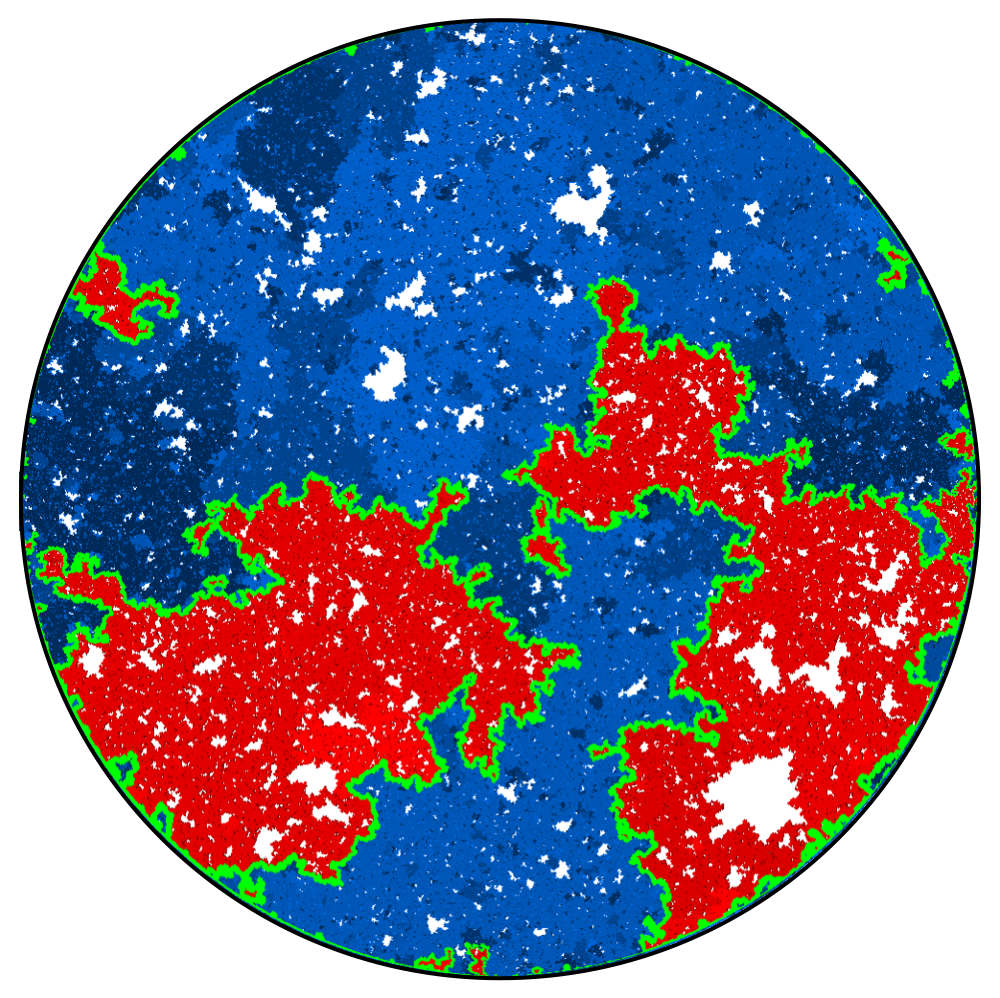}
	\caption{{\bf Top:} Simulation of the interface (in green) of the clusters of red $\CLE_6$ loops touching the left half-circle and the clusters of blue $\CLE_6$ loops touching the right half-circle for $p=1/4$. {\bf Bottom:} All of the boundary touching interfaces are shown.}
	\label{interface}
	\label{f1}
\end{figure}

An essentially equivalent way to formulate this result goes as follows: Consider a $\CLE_{\kappa'}$ in the upper half-plane, and look only at its outermost loops (the other ones will not matter here). We fix $p \in [0,1]$ and color each of the loops (and their interior) independently in red or blue with probability $p$ and $1-p$ respectively. We now have a coloring of the plane using two colors, and we can then look at the outer boundary of the closure of the union of the red connected components that touch the negative half-line. It turns out (see \cite{cle_percolations}) that it consists of the negative half-line together with some simple curve $\eta$ from $0$ to $\infty$, which is also on the outer boundary of the union of the blue connected components that touch the positive half-line (see Figure \ref {f1} for a simulation in the unit disk). It is furthermore shown in \cite{cle_percolations} that there exists $\rho \in [-2, \kappa-4]$ such that the law of $\eta$ is that of an $\SLE_\kappa (\rho; \kappa -6- \rho)$ process. The present paper will provide the explicit formula for $\rho$ as a function of $\kappa'$ and $p$, and therefore complete the identification of the law of $\eta$ (which was only known for $p=0$, $p=1/2$ and $p=1$):

\begin {theo}[The interfaces for $\CLE_{\kappa'}$ percolation processes]  
\label {thm:betarhoformula}
The relation between $p \in [0,1]$ and $\rho \in [-2, \kappa-4 ]$ when $\kappa' \in (4,8)$ is given by 
\[ p = \frac {\sin (\pi (\rho +2) / 2)}{\sin (\pi (\rho+2) /2 ) + \sin (\pi ((\kappa-6 -\rho) +2)/2)} \]
or equivalently 
$$ \rho + 2 = \frac 2 \pi \arctan \left( \frac {\sin ( \pi \kappa /2)}{1 + \cos (\pi \kappa /2) - (1/p)} \right).$$
\end {theo}

We wrote the right-hand side in this slightly strange form in order to stress the $\rho \leftrightarrow \kappa -6- \rho$ symmetry, and so that the $\sin(\cdot)$ terms take positive values and the angles belong to $[0, \pi]$. The fact that Theorem~\ref{thm:betarhoformula} implies Theorem~\ref{thm1} is a direct consequence of Theorem 7.10 of \cite{cle_percolations}. 
¨

To relate this formula with Section~\ref{S11}, one can recall that  the dimension of the intersection of an $\SLE_\kappa (\kappa-6-\rho)$ process 
with the real axis has been shown in \cite {MW_INTERSECTIONS} to be 
$$d (\kappa, \rho) =  1 - \frac {(\kappa -2 -(\rho+2) ) ( \kappa/2 -(\rho +2) )}{\kappa}. 
$$
Hence, this will be the dimension of the intersection of clusters of $\CLE_{\kappa'}$ loops (where each $\CLE_{\kappa'}$ loop is selected with probability $p \in (0,1)$) in the upper half-plane with the real line 
where $\rho +2  \in (0, \kappa-2)$ is given by Theorem~\ref{thm:betarhoformula}. 
One can note that when $p \to 0$, then $\rho +2 \to 0$, and one gets in the limit the dimension $2 - \kappa/2 = 2 - 8/ \kappa'$ of the intersection of one boundary-touching $\CLE_{\kappa'}$ loop with the boundary i.e., of the intersection of an $\SLE_{\kappa'}$ with the boundary, as one would expect. 
Statement~\ref{statement1} is then obtained by taking $\kappa'=6$ (i.e., $\kappa=8/3$) -- the exponent $a$ being $1- d$, and the formula for the FK$_2$-Ising model is obtained for $\kappa= 3$. 
We stress again that the formulas for exponents for those boundary critical exponents for these 
divide-and-color type models (also known as fuzzy Potts models 
\cite {mv1995fuzzypotts,hagg1999fuzzypotts}) in the discrete setting depend on $p$ in a very different way than one is accustomed to (as they here typically involve the $\arctan$ function), and that we presently know of no other way to derive such formulas than the one involving the LQG ideas that we will describe in this paper.  

Together with \cite {CLE_LQG}, the results of the present paper therefore completes the proofs of the 
statements announced in Section 7.4 of \cite {cle_percolations}. 

\subsection {Poissonian structure of $\CLE_{\kappa'}$ explorations on LQG surfaces}

The previous results will be obtained by understanding the L\'evy-type structures that emerge when one explores certain $\CLE_{\kappa'}$ decorated LQG surfaces for $\kappa' \in (4,8)$.
Many aspects of the arguments will mirror those of our 
 paper for $\CLE_{\kappa}$-decorated LQG surfaces \cite{CLE_LQG} (for $\kappa \in (8/3,4)$), that we will also directly refer to for an introduction and background. Just as in \cite{CLE_LQG}, all our arguments take place in the continuum and do not build on any considerations about random decorated planar maps, but the results do mirror some of the results that appear when one studies $O(N)$-models or FK-percolation models on planar maps via enumerative techniques, such as in \cite{BBG1,ccm2017lengths,bbck,budd,curienrichier}.  The Markovian structure that we unveil in the present paper can be viewed as the continuum counterpart on the peeling algorithms and their properties for these discrete models. This suggests of course a roadmap to identify their scaling limits in terms of CLE on LQG, using topologies related to these exploration mechanisms.

 The main philosophy of our results is the following: We consider a certain LQG surface (recall that this is a randomly chosen equivalence class of domains equipped with an area measure, boundary length measure, and metric, under an equivalence relation given by simple rules when applying conformal transformations -- the choice of this LQG surface involves a parameter $\gamma \in (0,2)$), and on this LQG surface, one samples an independent $\CLE_{\kappa'}$ for $\kappa'= 16 / \gamma^2$, and we color its loops independently into red and blue with respective probabilities $p$ and $1-p$. One chooses two boundary points, and then explores the red/blue interface that runs from one point to another. In other words, one follows the interface described in the previous paragraphs that is drawn on top of the independent LQG structure. 
 Together with this interface, one also discovers the encountered $\CLE_{\kappa'}$ loops that this interface meets (see Figure \ref {f2}), and one keeps track of (some aspects of) the connectivity properties of the remaining to be discovered surface. In the particular cases where $p =1$ and $p=0$, one just moves along the boundary of the surface and discovers the boundary-touching $\CLE_{\kappa'}$ loops.

\begin{figure}[pt!]
\includegraphics[width=0.48\textwidth]{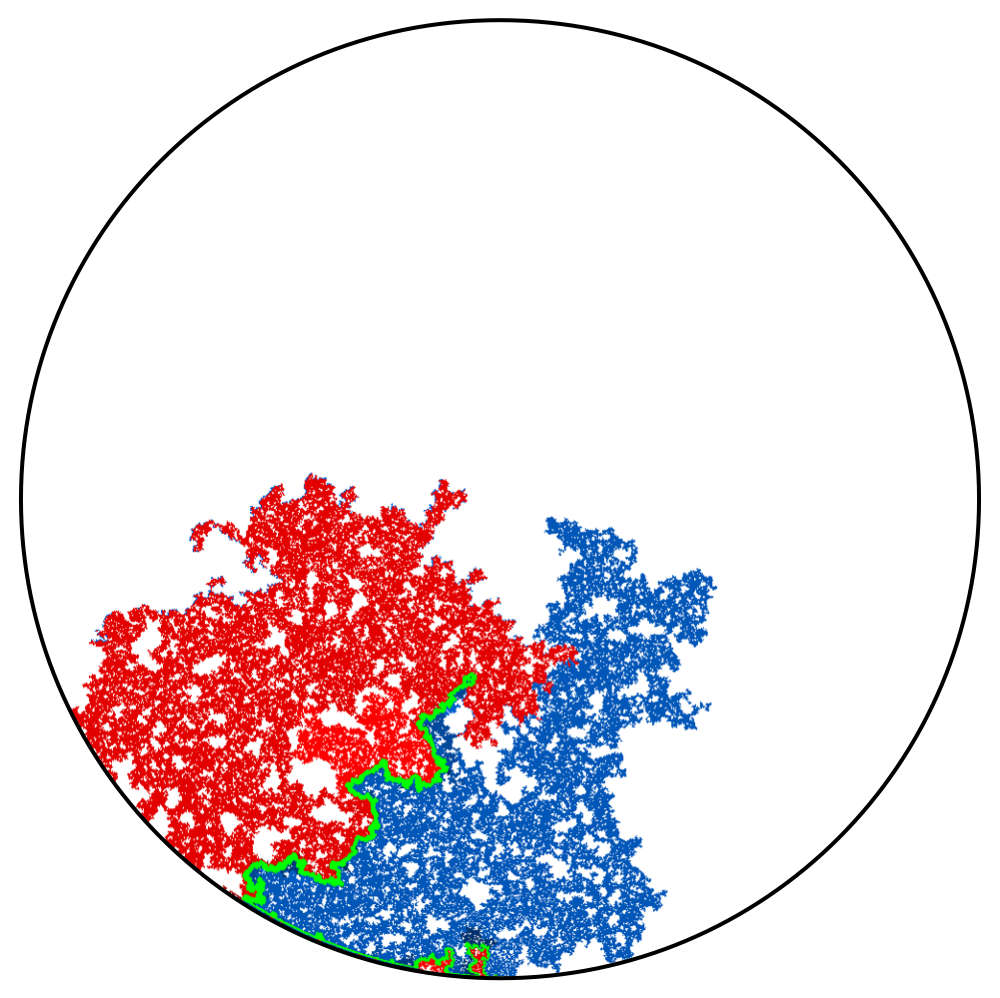}\hspace{0.02\textwidth}\includegraphics[width=0.48\textwidth]{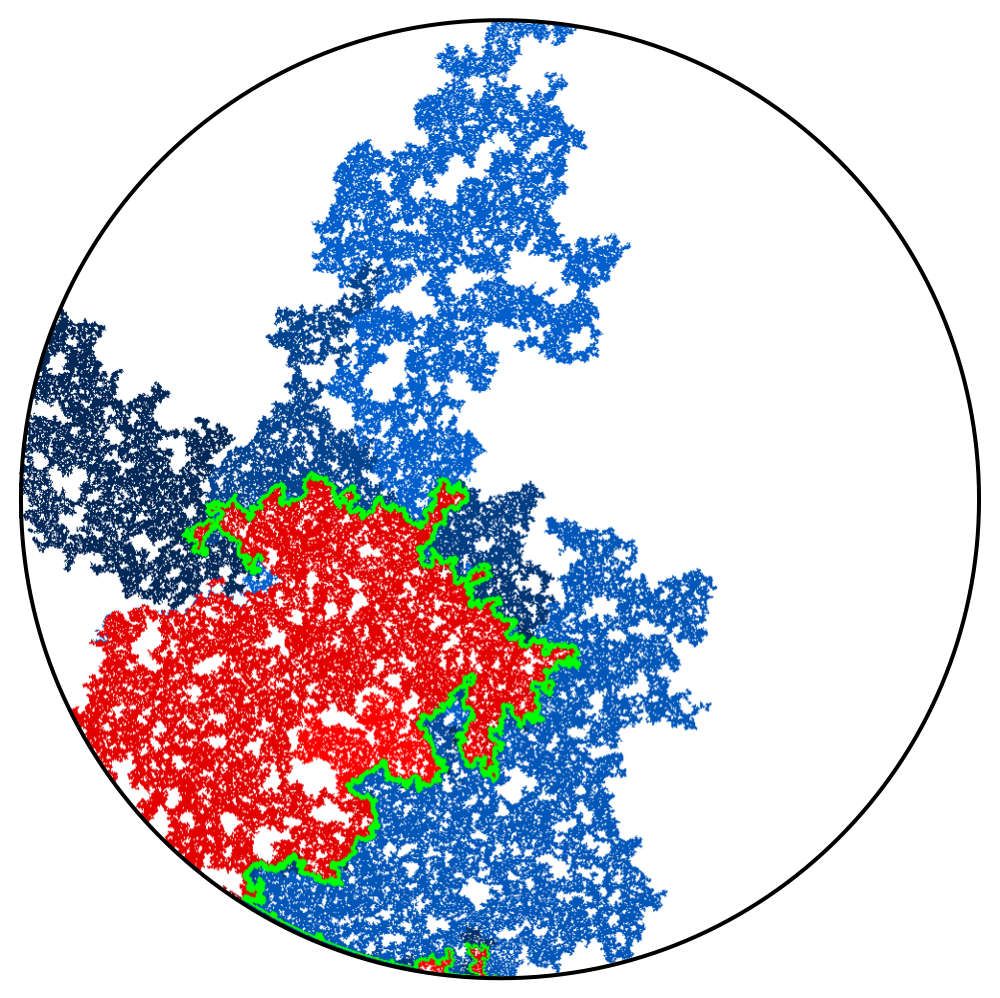}       
	\caption{The interface from the top of Figure~\ref{f1} together with loops which touch the interface at different times.  The discovered loops create an infinite chain of ``pockets'' through which the interface traverses.}
	\label{etat}
	\label{f2}
\end{figure}

\begin{figure}[pt!]
\includegraphics[width=0.48\textwidth]{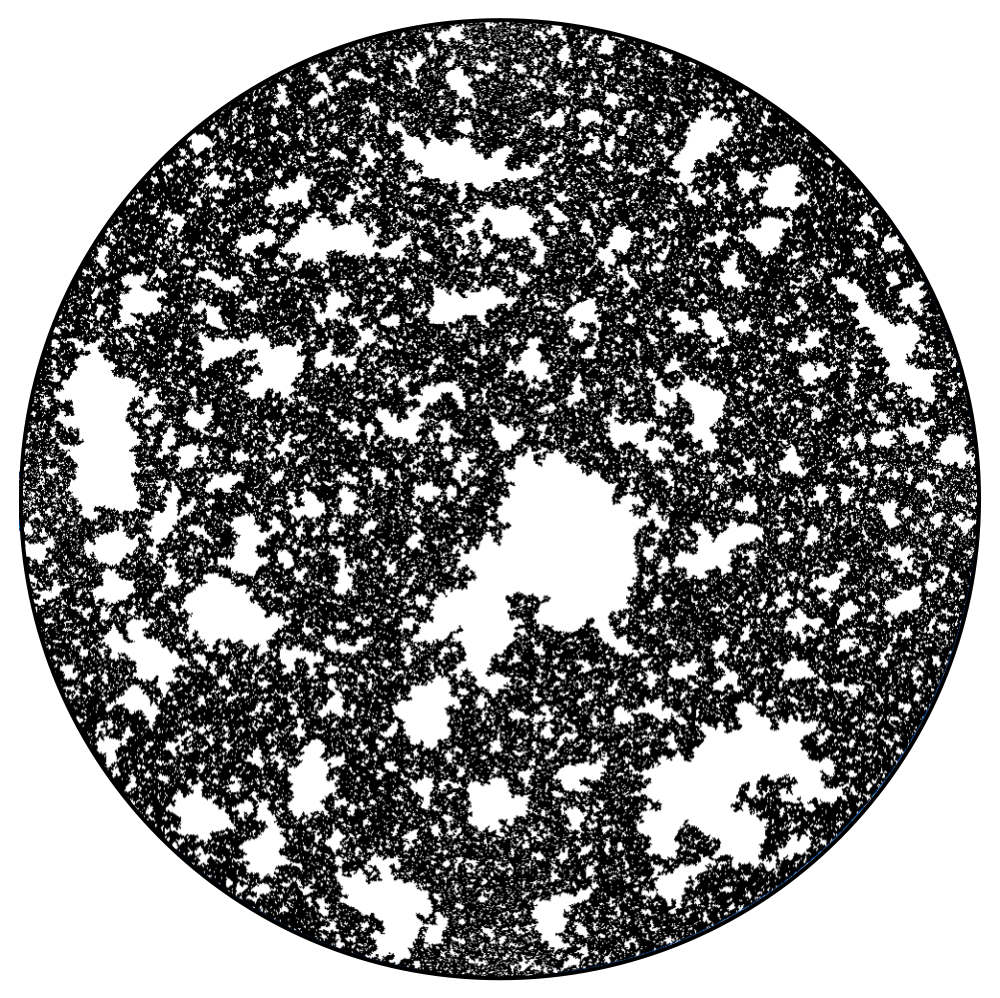}\hspace{0.02\textwidth}\includegraphics[width=0.48\textwidth]{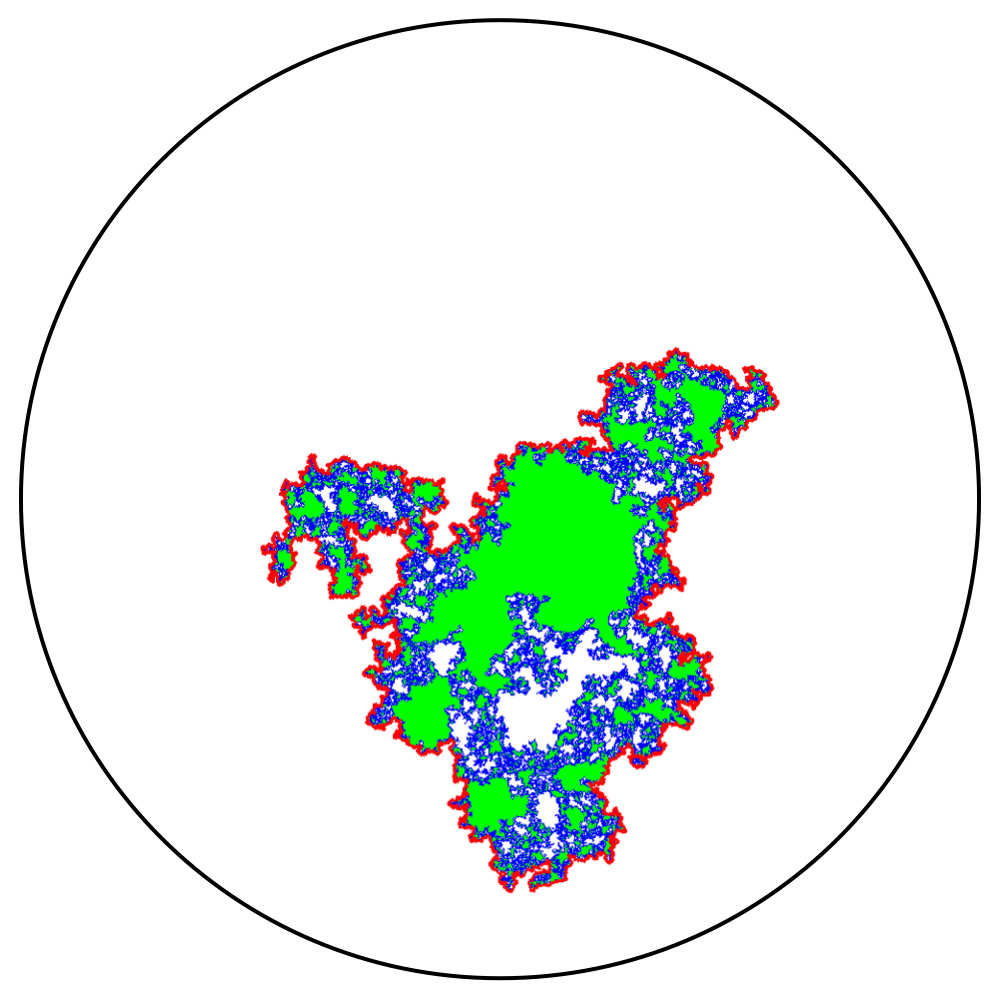}    
	\caption{{\bf Top:} $\CLE_6$ gasket.  {\bf Bottom:} Loop which contains the origin in a $\CLE_6$.  The bubbles have a natural tree structure which we will show is described by a stable looptree in the context of LQG.  Its outer boundary is an $\SLE_{8/3}$-type loop.}
	\label {f3}
\end{figure}

Then, our main statements (Theorems~\ref{thm3} and~\ref{thm4}) will be that if one starts with a well-chosen quantum surface (a so-called generalized quantum disk or a generalized quantum half-plane), this discovery process has a natural Markovian structure.  Let us illustrate this with the very special case $p=1$: Consider a so-called $\gamma$-quantum half-plane (the type of quantum surface which describes the local behavior of an LQG surface with boundary near a quantum typical boundary point) that we represent in the upper half-plane. Define  an independent $\CLE_{\kappa'}$ for $\kappa' = 16/ \gamma^2$ in the upper half-plane and consider the bi-infinite ordered family of $\CLE_{\kappa'}$ loops $(\CL_u)_{u \in U}$ that touch the boundary, ordered according to their left-most intersection point $x_u$ with the real line. Each $\CL_u$ then encircles a certain quantum surface $\CD_u$ (with a marked boundary point $x_u$). One consequence of our results will be that this bi-infinite ordered family of quantum surfaces is distributed like a Poisson point process of quantum surfaces (that we will refer to as generalized quantum disks). These Poisson point processes will be naturally related to stable processes, the properties of which will enable us to derive results such as Theorem~\ref{thm:betarhoformula}. 

It is worthwhile explaining already the differences between the results and proofs of the present paper and those of \cite{CLE_LQG}: 
\begin{itemize}
\item On the one hand, the exploration mechanisms, the L\'evy processes and the L\'evy trees that will be discussed in the present paper are in some sense simpler to understand in comparison to the ones appearing in \cite{CLE_LQG}. Roughly speaking, this corresponds to the fact that we will here be dealing with objects that are directly related to stable processes with index smaller than $1$ for which no L\'evy compensation mechanism is needed (whereas in \cite{CLE_LQG}, the index is in $(1,2)$). This is also related to the fact that one can explore a $\CLE_{\kappa'}$ by discovering all boundary touching $\CLE_{\kappa'}$ loops in the order in which one encounters them when one moves along the boundary, whereas this is not possible in the $\CLE_{\kappa}$ case. More generally, as shown in \cite{cle_percolations}, the $\CLE_{\kappa'}$ explorations are deterministic functions of the colored $\CLE_{\kappa'}$ that follow simple rules, while this is not the case for the $\CLE_{\kappa}$ explorations when $\kappa < 4$ (see \cite{msw2016notdetermined}). So, in this respect, the essence of the arguments in the present paper will be simpler than in \cite{CLE_LQG}. 
\item On the other hand, the LQG representation of these L\'evy trees is somewhat more complex. The main issue is to have a clear definition and understanding of the LQG surfaces that we are dealing with. The ones that do show up in the present paper are not simply connected domains, as opposed to the ones that appeared in \cite{CLE_LQG}. Roughly speaking, the type of quantum surfaces that appear naturally and that we will work in are the ones that correspond to the ``interior of a $\CLE_{\kappa'}$ loop''. When $\kappa \le 4$, the interior of a $\CLE_{\kappa}$ loop is a simply connected domain, whereas the interior of a $\CLE_{\kappa'}$ loop for $\kappa' \in (4,8)$ has infinitely many connected components (but we keep track of how they are connected within the loop, which corresponds to an additional tree-like structure) -- see Figure \ref {f3}. This gives rise to what we will call generalized quantum disks and generalized quantum half-planes (these objects have been referred to as forested disks or wedges in \cite{dms2014mating}). The reason for which such objects appear naturally in this setting is already clear from Figure~\ref{f2}, with its infinite chain of pockets in front of the interface. 
\end{itemize}

\subsection*{Outline}
The present paper is structured as follows:

\begin{itemize}
\item In Section~\ref{S2}, we first recall the definitions of the quantum surfaces that will be of interest in this paper (generalized disks and half-planes) and the results from the paper \cite{dms2014mating} that we will use in this paper. 
\item In Section~\ref{S3}, we study the case where one explores the 
boundary-touching $\CLE_{\kappa'}$ drawn on a quantum half-plane (this corresponds to the case $p=1$ mentioned above), and see how $(4 / \kappa')$-stable L\'evy processes and Poisson point processes of quantum disks show up naturally when one explores a generalized quantum half-plane. 
\item In Section~\ref{S4}, we study the case $p \in [0,1]$, and derive the first main result of the present paper (Theorem~\ref{thm3}), about CLE explorations of generalized quantum half-planes. 
\item Then, in Section~\ref{S5}, following ideas that were already developed in \cite{CLE_LQG}: 
\begin{enumerate}[(i)]
\item  We will explain what happens when one explores a colored $\CLE_{\kappa'}$ drawn on a generalized quantum disk (instead of a generalized quantum half-plane), deriving Theorem~\ref{thm4} (which is the counterpart of Theorem~\ref{thm3} in that case), and how this relates to a 
fragmentation L\'evy tree.
\item We complete the proof of Theorem~\ref{thm:betarhoformula}.
\item Finally, we will mention how the description in (i) allows us to define the ``natural LQG measure'' in the $\CLE_{\kappa'}$ gasket.
\end{enumerate}
\end{itemize}

\section {LQG preliminaries} 
\label {S2}

We review some features and results about LQG surfaces that we will use. We first discuss the ``usual'' quantum surfaces (disks and wedges), then the ``generalized ones'' (a.k.a. forested surfaces), and we recall some of the ``slicing/welding'' results from \cite {dms2014mating} that will be instrumental in the present paper.
In order to make this part digestible, we chose not to give the precise definitions of the various objects that will be discussed (quantum disks, quantum wedges, $\SLE_\kappa (\rho_1;\rho_2)$ processes); instead, we discuss the actual properties that will be of use in the present paper and refer to other papers for the actual definitions.

\subsection {Quantum surfaces}
Unless otherwise specified, in the remainder of the paper,  $\kappa' \in (4,8)$. The values $\kappa$, $\gamma$, $\alpha$ and $\alpha'$ are related to $\kappa'$ by: 
\begin{equation}
\label{eqn:relations}
\kappa = \frac {16}{\kappa'},  \quad \gamma = \sqrt {\kappa}, \quad \alpha = \frac {4}{\kappa}, \quad \alpha' = \frac {4}{\kappa'}. 
\end{equation}
In particular, $\alpha \in (1,2)$ and $\alpha' \in (1/2, 1)$ (and $\alpha'$ stable subordinators exist while $\alpha$-stable subordinators do not exist).  All the LQG surfaces that we will consider will be $\gamma$-LQG surfaces, i.e., they correspond to  the exponential of $\gamma h$, where $h$ is a variant of the GFF (i.e., typically, the GFF plus some harmonic function). The corresponding area and boundary length measures exist for all $\gamma \in (0,2]$ and can be rigorously defined via a regularization procedure \cite{hk1971quantum,kahane1985gmc,DS08}.  The metric has also been defined for $\gamma=\sqrt{8/3}$ in \cite{qlebm,qlebm2,qlebm3} and for all $\gamma \in (0,2)$ in \cite{dddf2019tightness,gm2019exunique}.  We also recall that an LQG surface (with or without marked points) can be viewed (and this is the perspective we will use in the present paper) as an equivalence class of domains (equipped with such an area measure, or equivalently with an instance of a variant of the Gaussian free field) under conformal maps.  That is, two domain field pairs $(D,h)$, $(\wt{D},\wt{h})$ are said to be equivalent as quantum surfaces if there exists a conformal transformation $\varphi \colon D \to \wt{D}$ such that $h = \wt{h} \circ \varphi + Q \log|\varphi'|$ where $Q = 2/\gamma+\gamma/2$.  This definition naturally generalizes to the setting in which one keeps track of extra marked points. 

Let us first recall a few facts about the usual LQG surfaces that also appeared in \cite{CLE_LQG}: 
\begin{enumerate}[(i)]
\item The first fundamental building block is the standard quantum disk (we will from now on drop the reference to $\gamma$).  Recall that (in some sense made precise in \cite{CLE_LQG}), they correspond to the quantum surfaces that are encircled by a $\CLE_\kappa$ loop in an ambient (infinite volume) GFF. These quantum disks also come equipped with a boundary length measure (that is a function of its area measure), and have a finite total boundary length $l$ as well as a finite total area $\CA$.  We will denote by $P_l$ the probability measure on quantum disks with a prescribed length $l$.  

One can obtain a quantum disk with one (resp.\ two) marked boundary point (resp.\ points) by choosing this point (resp.\ these points) uniformly (resp.\ uniformly and independently) with respect to its boundary measure.  In particular, we note that the law of a marked quantum disk is invariant under shifting the marked point by some fixed amount of boundary quantum length.

\item The so-called thin quantum wedge of weight $W_D:= \gamma^2 -2 $ can be viewed as an infinite ordered family of doubly marked quantum disks, that is defined as a Poisson point process $(C_{t_i}, a_{t_i}. b_{t_i})$ with intensity $dt \otimes (\int_{\R_+} l^{-\alpha} P_l dl )$ on $[0, \infty)$. One way to think about it is as an infinite chain of beads (each disk $C_{t_i}$ being one bead attached to the rest of the chain via their marked points). One can define the quantum boundary measure on the boundary of this wedge by adding up the quantum boundary measures of the disks, and one can note that the total boundary length of the disks $C_{t_i}$ with $t_i < t$ is almost surely finite.

\item When one looks at LQG surfaces with two marked boundary points, the more general class of surfaces that appear are the quantum wedges. 
For each $W \ge \gamma^2 /2 $, a (thick) quantum wedge of weight $W$ is an equivalence class of quadruples $(D, h, a, b)$, where $a$ and $b$ are now two boundary points 
(one of which is the apex of the wedge, and the other one is the ``point at infinity''). Again, one can define a boundary length measure on $\partial D$, which is locally finite, except in the neighborhood of ``infinity''. The weight $W=2$ plays a very special role, and we will refer to it as the quantum half-plane. This is the case where the apex is in fact a ``boundary-typical'' point: If $a'$ is obtained from $a$ by moving a fixed amount of boundary length to the right or to the left, then the new quantum surface $(D,h, a', b)$ is still a quantum half-plane. One way to explain this feature is 
that the half-plane is what one observes when one zooms into the infinitesimal neighborhood of a boundary-typical point of any type of quantum surface. 

When $W \in (0, \gamma^2 /2)$, the natural object to consider is a so-called {\em thin} wedge of weight $W$. Just as in the case where $W=W_D$ above, it is an infinite chain formed by a Poisson point process of quantum surfaces (called the beads of the thin wedge) with finite boundary length and two marked boundary points. The Poisson point process of the boundary lengths of the beads of a thin wedge of weight $W$ has intensity $dl/l^{2-2W/\gamma^2}$.
\end{enumerate}

One definition of quantum wedges uses an encoding via excursions of Bessel processes away from $0$, or equivalently, excursions away from $-\infty$ of drifted Brownian motion. The weight $W$ is then related to the dimension of the Bessel process or to the drift of the Brownian motion (the difference between thick and thin wedges corresponds then to the sign of the drift).

We are now ready to state \cite[Theorem~1.4]{dms2014mating} that we will use here. We assume here that $\gamma \in (\sqrt {2}, 2)$ and that $\kappa \in (2, 4)$. Here (and throughout this paper), an $\SLE_\kappa (\rho_1; \rho_2)$ will denote a process with two marked points immediately to the left and right sides of the starting point of the curve.
\begin{theo}
\label{thm:gluing_wedges}
Fix $W > 0$ and suppose that $\CW$ is a quantum wedge of weight $W>0$.  Let $\rho_1,\rho_2 > -2$ be such that $W = W_1 + W_2$, where $W_1 = \rho_1 + 2$ and $W_2 = \rho_2 +2$.  Let $\eta$ be an independent $\SLE_\kappa(\rho_1;\rho_2)$ process from the origin point to the infinity point in $\CW$ (if $W$ is a thin wedge, it is the concatenation of such processes in each of the beads).  Then the surfaces $\CW_1$ and $\CW_2$ which respectively correspond to the part of $\CW$ which is to the left and right of $\eta$ are independent quantum wedges with weights $W_1$ and $W_2$ (again, these can be thin wedges if $\eta$ hits the boundaries of $\CW$).  
\end{theo}

In some sense, to understand the arguments in the present paper, this ``additivity/divisibility'' property is the only feature that one needs to have in mind (together with the scaling property of boundary lengths of beads mentioned just above). An additional fact proved in \cite {dms2014mating} but that we will not use here, is that $\eta$ and $\CW$ are almost surely determined by $\CW_1$ and $\CW_2$ (this means that ``welding two wedges of weight $W_1$ and $W_2$ provides a wedge of weight $W_1+W_2$''). 

One simple instance of the theorem is when $W =4$ and $\eta$ is a $\SLE_\kappa$ (i.e., $\rho_1 = \rho_2 = 0$). Then, $\eta$ divides $\CW$ into two independent quantum half-planes (i.e., quantum wedges of weight $2$). Actually, it is known (this is the original quantum zipper result from \cite {SHE_WELD}) that if $\eta$ is parameterized according to its quantum length, then for each $t$, the domain $(\CW \setminus \eta [0,t], h, \eta(t), \infty)$ is wedge of weight $4$.

\subsection {Forested wedges, generalized quantum disks and half-planes}

Let us now provide some background on the {\em generalized quantum disks} and their variants, which are referred to as forested wedges in \cite{dms2014mating}. A major role will be played in the present paper by these generalized quantum disks that (as we will actually show) can in some sense be viewed as the surface that is ``inside'' of an $\SLE_{\kappa'}$ loop in an ambient LQG surface. 

\begin {rema} 

We will use the following terminology to clearly make the difference between the LQG structures that have the topology of the disk (or the sphere), and the ones with bottlenecks: For a given $\gamma \in (0, 2)$, the natural quantum length of an $\SLE_{\kappa'}$-type (non-simple) curve drawn on a $\gamma$-LQG surface will be referred to as its {\em generalized} LQG length. Similarly, the LQG surfaces with special symmetries that we will define in this section and that correspond to surfaces with $\SLE_{\kappa'}$ outer perimeter, will be called {\em generalized} quantum disks and half-planes. These LQG surfaces will then have a {\em generalized quantum boundary length}.   
These surfaces with bottlenecks have appeared in the literature under various names (forested surfaces,  pinched surfaces, beaded surfaces, surfaces with baby universe, touching random surfaces, KPZ with the other gravitational dressing, etc., see for instance \cite {JM,klebanov}). 

\end {rema}

 Recall that the boundaries of the bounded connected components of the complement of an $\SLE_{\kappa'}$ loop are  
$\SLE_{\kappa}$-type loops, so that if we view the loop as drawn on a $\gamma$-LQG surface, each one of these bounded components $O_j$ will be (similar to) a standard quantum disk. 
The generalized quantum disk loosely speaking corresponds to the collection of all these quantum surfaces $O_j$ together with the knowledge of how they are ``connected'' within the $\SLE_{\kappa'}$ loop (this connectivity equips naturally the family of these connected components with a tree structure).  

To properly define these generalized disks, one first defines the measure that determines the tree structure: This is the measure on $\alpha$-stable looptrees  defined in \cite{curienkortchemski1} (we refer to this paper for a detailed description of these structures).
The looptree is defined out of the excursion of an $\alpha$-stable L\'evy process with no negative jumps (which is defined under an infinite measure). The idea is then to associate to this excursion the 
usual tree structure as introduced by Le Gall and Le Jan \cite{LGLJ}, except that the nodes of the tree (which in the usual stable tree correspond to the jumps of the L\'evy excursion) will be given a circular structure with length given by the jump size. More precisely, if $X \colon [0,T] \to \R_+$ is the excursion, one defines an equivalence relation on the graph $\{(t,X(t)) : t \in [0,T]\}$ of $X$ by saying that $s \sim t$ if and only if $X(s) = X(t)$ and the horizontal chord connecting $(s,X(s))$ and $(t,X(t))$ lies below the graph of $X|_{[s,t]}$.  If $t$ is a jump time of $X$, then we also declare that $(t,X(t))$ and $(t,X(t^-))$ are equivalent (which produces the circular structure, i.e., each jump of $X$ therefore corresponds to a topological circle in the quotient $\CT$).  We note that $\CT$ is naturally rooted via the projection $\rho$ of the origin $(0,0)$. 

The way to think about it is that the looptree will encapsulate the information on the boundary lengths of the various $O_j$'s and how they are connected towards the root. 

Some features of looptrees: 
\begin{itemize}
\item When $a$ and $b$ are two points on the looptree, we have a unique chain of loops that connects them. The sum of the lengths of these loops is finite, and so is the shortest path joining these points. So, one has a natural measure on the boundary of this chain, and a distance on the looptree.  
\item There is also a natural notion of boundary length of the entire looptree.  It is very simple to see that the sum of the lengths of the loops in a looptree is infinite (as the sum of the jumps which occur in any non-empty open interval of time of an $\alpha$-stable L\'evy process with $\alpha \in (1,2)$ is infinite). However, one can make sense of the natural (and finite) measure living on the boundary of the generalized disk, for instance by taking the image of the Lebesgue measure on $[0,T]$ in the construction above (this also corresponds to the fact that  the Hausdorff dimension of the looptree with respect to the aforementioned distance is $\alpha$ as shown in \cite{curienkortchemski1}). We will refer to this measure as the generalized boundary length measure of the looptree. 
\item It turns out that the root of the looptree is a boundary-typical point, in the sense that if one resamples the root uniformly on the boundary according to this generalized boundary length, one does not change the law of the looptree (this was established by Curien-Kortchemski \cite{curienkortchemski1} as a byproduct of their discrete to continuum scaling limit result,  see also Duquesne-Le Gall \cite{dl2002levy_trees} as well as Archer \cite{Archer} for a continuum proof).
\end{itemize}

A (marked) generalized quantum disk is then obtained from an $\alpha$-stable looptree by assigning a conformal structure to each of the loops using a (standard) independent quantum disk (with $\alpha, \gamma$ matched as in~\eqref{eqn:relations}) with boundary length given by the length of the loop, and that is marked at the point connected to the root. 

A consequence of the rerooting property of the looptree is that the measure on marked generalized disks is also invariant under re-rooting.  Indeed, in terms of the quantum disk structure, the rerooting operation for the generalized disk corresponds to shifting the marked points of each disk so that they fall along the branch on the disk to the root.

\begin {rema} 
 \label {areascaling} 
If we consider a generalized quantum disk with generalized boundary length $\ell$, then its total quantum area $\CA_{\ell}$ has a finite expectation, and its law is equal to that of $\ell^{2/\alpha} \CA_1$.  If we instead consider a (usual) quantum disk with (usual) quantum boundary length $\ell$, then its total quantum area $\CA_\ell$ has finite expectation and its law is equal to that of $\ell^2 \CA_1$.
\end {rema}

 If we choose two points according to the generalized boundary length measure, then one has a forested spine decomposition with a PPP of other looptrees glued to the spine with respect to the aforementioned boundary measure on the spine.

If we replace the spine by an infinite Point process of loops with boundary lengths intensity $dl / l^{\alpha}$ on $\R_+$ and then associate with each loop the conformal structure given by that of a quantum disk then one obtains a thin quantum wedge $W_D$.  If we then add a PPP of generalized quantum disks on the left and right sides of the boundary, then we get the structure that we will call {\em the generalized quantum half-plane}.

One useful way to think about the generalized half-plane (and that can be made precise) is that it is the structure that one obtains when zooming in the neighborhood of a boundary-typical point of a generalized disk, chosen according to the generalized boundary length measure. 
Again, the generalized boundary length of a quantum half-plane is locally finite.  Moreover, it follows from the root invariance of the generalized quantum disk that the generalized quantum half-plane is invariant under the operation of shifting the root by a fixed amount of generalized boundary length.

\begin {rema} 
The natural infinite measure on quantum disks (that corresponds to the jump measure of the stable subordinator) is $l^{-\alpha-1} P_l dl$, where $P_l$ is the probability measure on disks with boundary length $l$. The reason why in the spine decompositions of generalized half-planes, the measure $dl l^{-\alpha} P_l$ shows up instead (see for instance already in the definition of the quantum wedge of weight $W_D$ in the previous section) can be interpreted by the fact that being on the spine provides a size-biased PPP with an additional factor proportional to the boundary length. 
\end {rema}

The operation of gluing an independent Poisson point process of generalized quantum disks on the boundary of a (usual) quantum surface is referred to as {\em foresting} in \cite {dms2014mating}. In particular, it is possible to forest quantum wedges of other weights than $W_D$.

\begin {rema} 
 \label {scalingremark} 
Suppose that one considers a usual quantum surface of boundary length $l$, and that one attaches to its boundary a Poisson point processes of generalized disks. Then, the sum of the generalized boundary lengths of these disks will be the value of an $\alpha'$-stable subordinator at time $l$, which has the law of $l^{1/\alpha'}$ times the value of this subordinator at time $1$. 

In particular, this implies that if we consider a thin wedge of weight $W$, and then take its forested version and denote by $\ell_{t_i}$ the total generalized boundary lengths of the generalized disks attached to each bead $B_{t_i}$, then $(\ell_{t_i})$ is a Poisson point process of intensity $d\ell / \ell^{1 + (1- 2W/ \gamma^2)\alpha'} = d\ell / \ell^{1+ \alpha'- W/2}$. For instance, for $W = \gamma^2 -2$, we get $d\ell/ \ell^{2-\alpha'}$.
\end {rema}

The counterpart of Theorem~\ref{thm:gluing_wedges} for forested wedges uses $\SLE_{\kappa'} (\rho_1'; \rho_2')$ and can be stated as follows: 

\begin{theo}
\label{thm:gluing_wedges'}
Let $\CW$ be a forested wedge of weight $W \ge 2 - \gamma^2 / 2$. Suppose that $W_1, W_2 \ge 0$ with 
$W_1 + W_2 + (2 - \gamma^2 / 2 ) = W$. 
We then define $\rho_1',\rho_2'$ so that $ W_i = \gamma^2-2 + \gamma^2\rho_i' /4$ for $i=1,2$, and let $\eta'$ be an independent $\SLE_{\kappa'}(\rho_1';\rho_2')$ process from the point at infinity to the origin of the wedge (when the wedge is thin, then it is a concatenation of such processes -- one in each bead of the spine of $\CW$).  Then the generalized quantum surfaces $\CW_1$ and $\CW_2$ which consists of the components of $\h \setminus \eta'$ which are to the left (resp.\ right) of $\eta'$ (when viewed from the origin) are independent forested quantum wedges of weight $W_1$ and $W_2$.
\end{theo}

\begin {rema} 
\label {sideremark}
In general, when one considers an $\SLE_\kappa (\rho_1;\rho_2)$ process, the convention is that $\rho_1$ and $\rho_2$ respectively correspond to the intensity of the drift due to the marked points that are to the left and to the right of the tip of the curve (viewed from this tip). However, in the special case where we are looking at an $\SLE_{\kappa'}(\rho_1';\rho_2')$ started at the point at infinity  (like in Theorem \ref {thm:gluing_wedges'}), we will use the convention that $\rho_1'$ corresponds to the force point located ``to the right'' of the curve when viewed from the tip of the curve, so that this becomes the left when viewed from the target point. For instance, when $\rho_1'$ gets very close to $-2$, this process will tend to come down from infinity along the {\em negative} real axis.

At various instances in the present paper, we will similarly use the following terminology for left and right boundaries. In the imaginary geometry context, when one is looking at an $\SLE_{\kappa'}$-type curve $\eta'$ from $a$ to $b$, then it is natural to describe its outer boundaries as $\SLE_\kappa$-type curves from $b$ to $a$. We will refer to the left and right boundaries of $\eta'$ as the curves that lie to its left and to its right respectively, {\em when viewed from $b$ to $a$}.

All these conventions are for instance already used in earlier work in the imaginary geometry framework -- see for instance \cite[Figure~2.5]{MW_INTERSECTIONS}, that also illustrates why such a ``side-switching convention'' is useful.
\end {rema}

Again, it is in fact possible to reconstruct $\eta'$ and $\CW$ from $\CW_1$ and $\CW_2$. 

A special case of Theorem~\ref{thm:gluing_wedges'} is when $W=2$ and $\eta'$ is an $\SLE_{\kappa'}(\kappa' -6)$. In that case, the two surfaces $\CW_1$ and $\CW_2$ are forested wedges with weights $2-\gamma^2/2$ and $\gamma^2-2$ -- in particular, $\CW_2$ is a generalized quantum half-plane.

One way interpret and actually derive Theorem~\ref{thm:gluing_wedges'} is to view it as a decomposition of the (non-forested) wedge of weight $W$ into three non-forested independent wedges of weight $W_1$, $2- \gamma^2 / 2$ and $W_2$ that are separated by the left-boundary $\eta_L$ and by the right-boundary $\eta_R$ of $\eta'$ (in the sense explained in Remark \ref {sideremark}). 

Indeed, the imaginary geometry coupling interpretation of the curves $\eta'$, $\eta_R$ and $\eta_L$ viewed as flow lines/counterflow lines of an auxiliary GFF show that \cite[Section~7]{MS_IMAG}
\begin{lemm}
\label{lem:sle_outer_boundary}
Suppose that $\eta'$ is an $\SLE_{\kappa'} (\rho_1'; \rho_2')$ in the upper half-plane from $\infty$ to $0$.  Then: 
\begin{enumerate}[(i)]
\item The law of its left boundary $\eta_L$ is that of an $\SLE_\kappa(\kappa-4+(\kappa \rho_1' /4)  ; ({\kappa}/{2}) -2 + (\kappa \rho_2' / 4))$ from $0$ to $\infty$.
 \item The conditional law of its right boundary $\eta_R$ given $\eta_L$ is that of an $\SLE_{\kappa}(-\kappa / 2; \kappa-4 + (\kappa \rho_2' /4))$ in the domain to the right of $\eta_L$. 
 \item The conditional law of $\eta'$ given $(\eta_R, \eta_L)$ is an $\SLE_{\kappa'}(\kappa'/2-4;\kappa'/2-4)$ process in the beads squeezed between $\eta_R$ and $\eta_L$.
\end{enumerate}

\end{lemm}

Theorem~\ref{thm:gluing_wedges'} is then obtained by successively applying 
Theorem~\ref{thm:gluing_wedges} to $\eta_L$ in ${\mathcal W}$, then to $\eta_R$ in the quantum wedge that lies to the right of $\eta_R$, and then finally applying the following result to $\eta'$ in the middle thin wedge of  weight $2 - \gamma^2 /2$ that lies between $\eta_R$ and $\eta_L$ \cite[Theorem~1.15]{dms2014mating}: 

\begin{theo} 
\label{thm:gluing_forested_lines}
Suppose that $\CW$ is a quantum wedge of weight $2-\gamma^2/2$ and that $\eta''$ consists of a concatenation of independent $\SLE_{\kappa'}(\kappa'/2-4;\kappa'/2-4)$ processes, one for each bead of $\CW$.  Then the bubbles which are to the left (resp.\ right) of $\eta''$ (when parameterized via the quantum natural length of $\eta''$) are two independent Poisson point processes of generalized quantum disks.
\end{theo}

\section {The case $p=1$} 
\label{S3}
For presentation purposes, we choose to first present some of the results and proofs for the totally asymmetric case where $p=1$. In this case, the interface $\eta$ follows the boundary of the domain so that determining its law is not an issue, but the identification of the stable processes requires some non-trivial input. This will allow us to explain some of the ideas that will then be used again in the general case $p \in [0,1]$. 

\subsection {The setup and the first main statement} 

\label {S3.1}

Let us consider a generalized quantum half-plane $\CH$, where $x(0)$ is its marked boundary-typical point. We let $\CW$ be the usual quantum wedge of weight $W_D= \gamma^2 -2$ consisting of a chain of quantum disks from $x(0)$ to infinity in $\CH$. As explained above, if we condition on $\CW$, the remaining surfaces in $\CH \setminus \CW$ can be viewed as a Poisson point process of generalized disks glued to the boundary of $\CW$ according to its boundary length measure.

\begin{figure}[ht!]
\includegraphics[scale=.7]{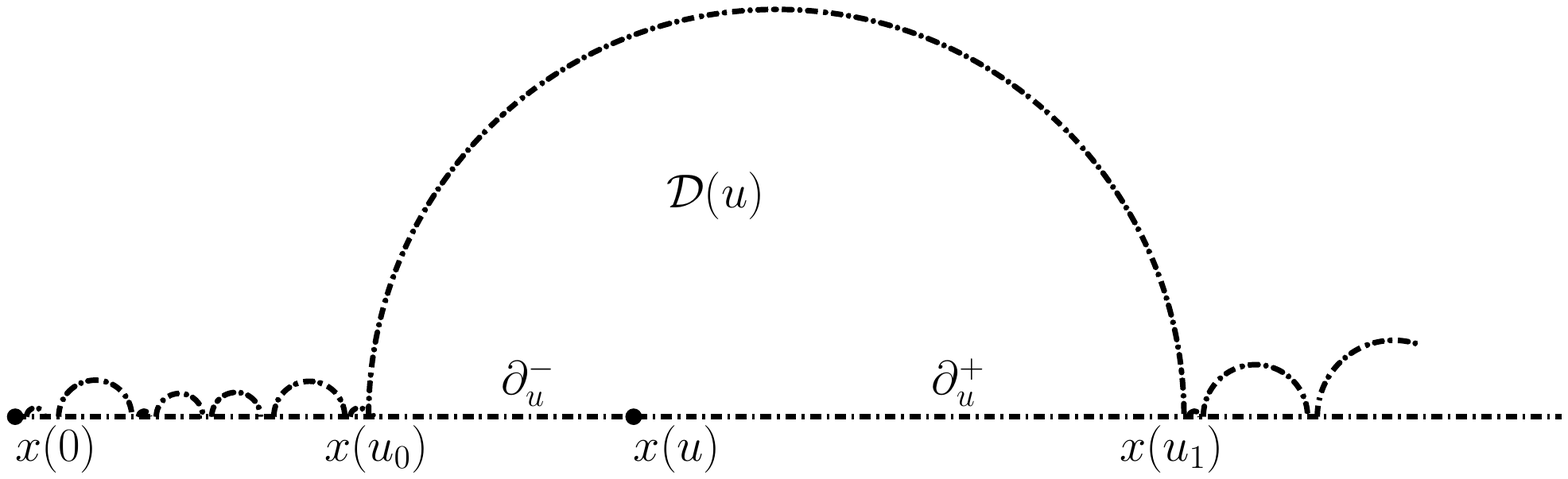}  
	\caption{A representation of the quantum wedge $\CW$ of weight $W_D$ as a chain of Euclidean half-disks -- the generalized half-plane is obtained by foresting the dash-dotted boundaries.}
	\label {f4}
\end{figure}

For each fixed $u >0$, we define $x(u)$ to be the boundary point of $\CW$ that lies at $u$ units of quantum boundary length from $x (0)$ on the counterclockwise boundary arc of $\CW$ starting at $x(0)$. We can note that for a given $u>0$, $x(u)$ will almost surely lie on the boundary of a quantum disk of $\CW$, that we denote by $\CD(u)$ (see Figure \ref {f4}).

For each given $u$, we define $u_0 = u_0(u)$ and $u_1=u_1(u)$ so that $x(u_0)$ and $x(u_1)$ are the first and last point of $\partial := \{x(v), v  \in [0, \infty) \}$ that lie on the boundary of the quantum disk $\CD(u)$.  Let $\partial_u := \{ x(v) : v \in [u_0(u),u_1(u)]\}$.  We call $\CW^+(u)$ the ordered collection of disks in $\CW$ that are ``between'' $\CD(u)$ and infinity ($\CD(u)$ not included, so this is the family of all $\CD(v)$ for all rational times $v > u_1$). It is easy to see from this definition that $\CW^+ (u)$ (with marked point at $x(u_1)$) is also a wedge of weight $W_D$.

We will denote by $\CC$ the set of times $u$ corresponding to points ``in between'' the beads. This is the fractal set obtained by removing from $\R_+$ all intervals of the type $(u_0(u), u_1(u))$ for rational times $u$.

We now sample a $\CLE_{\kappa'}$ in $\CW$, i.e., an independent $\CLE_{\kappa'}$ inside each of the disks forming $\CW$. 
We fix $u>0$, and for what will immediately follow, only the $\CLE_{\kappa'}$ loops in $\CD(u)$ will matter. We now  define $\CD_0(u)$ from $\CD(u)$ as follows: We first remove from $\CD(u)$ all the $\CLE_{\kappa'}$ loops (and their interiors) that intersect the set $\partial_u^- := \{x(v), v  \in [u_0, u ] \}$. 
In the remaining set, we look at the ordered chain of connected components whose boundary intersects $\partial_u^+:= \{x(v), v  \in [u, u_1 ] \}$ and we view this chain as a chain of quantum surfaces -- each with a pair of marked points on $\partial_u$. We denote this chain by $\CW^0(u)$. 

To illustrate what follows, we can already state the following (which is in fact a consequence of Proposition~\ref{main1} that we will state and prove below). 
\begin{prop} 
\label{firstprop}
 The concatenation $\CW(u)$ of the two chains $\CW^0 (u)$ and $\CW^+ (u)$, with marked point at $x(u)$, is also a quantum wedge of weight $W_D$. 
\end{prop}

In fact, we will exhibit a generalized quantum half-plane $\CH(u)$, that will contain this wedge $\CW(u)$ as its spine-wedge. 
We are going to define $\CH(u)$ in several steps: 

(1) First, we endow the set of connected components that form $\CD_0 (u)$ with a tree structure. For this, we use the connectivity structure of the ``outside'' of the $\CLE_{\kappa'}$ loops that one has removed from $\CD(u)$ to define $\CD_0(u)$. We can then use the LQG structure of each of the connected components of $\CD_0(u)$, and view $\CD_0(u)$ as a tree of LQG surfaces. Note also that the boundary of each of these components carries a quantum length measure, so that we can in fact view it as a loop-tree of LQG surfaces. 
This tree will contain the chain of components that form $\CW_0 (u)$. On part of the boundary of $\CW_0 (u)$, the trees  created by the fjords of the CLE$_{\kappa'}$ loops are grafted. 

\begin{figure}[ht!]
\includegraphics[scale=.7]{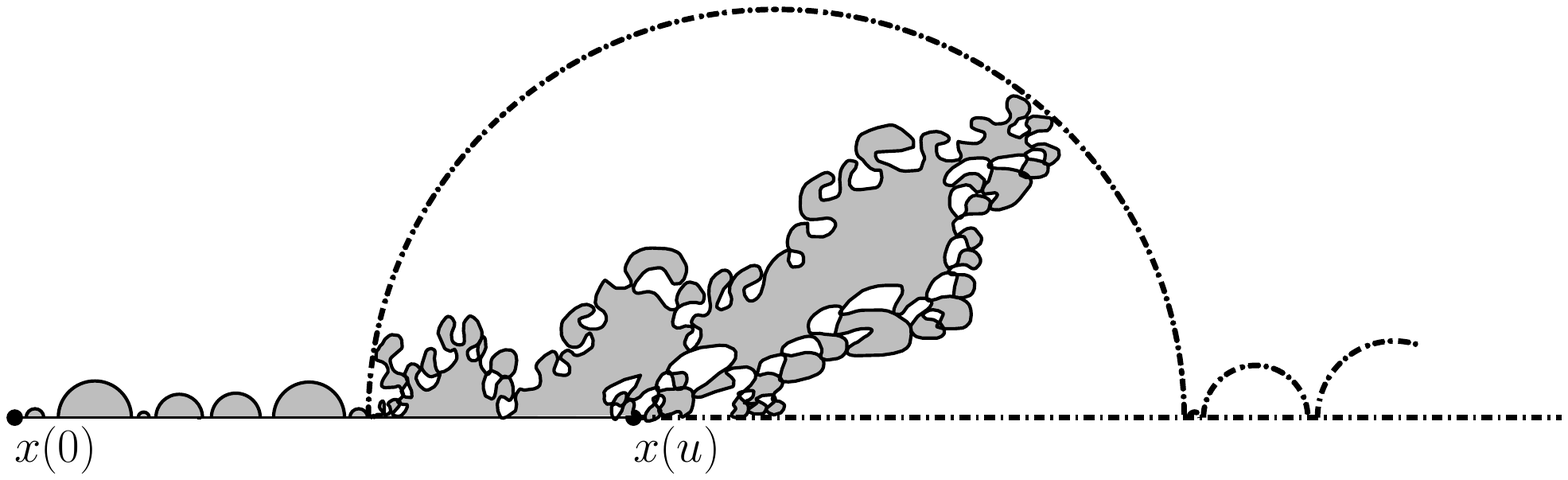}   
	\caption{A representation of situation at time $u$: The generalized half-plane ${\mathcal H}(u)$ is obtained by foresting the dash-dotted boundaries of the non-shaded region (and adding them to the non-shaded region). The wedge $\CW(u)$ is the union of the white connected components that have subintervals of  $[x(u), \infty)$ on their boundary (i.e., one cuts off the fjords of the non-shaded region).}
	\label{f5}
\end{figure}

We now define the loop-tree of quantum surfaces $\CH(u)$ that is obtained by extending $\CD_0(u)$ using the following three additional grafting/foresting operations: 

(2) On $\partial \CD(u) \setminus \partial_u$, one also adds the same generalized disks as in $\CH$.

(3) On $\partial_u^+$, one adds the same Poisson point process of generalized disks as in $\CH$. 

(4) At $x(u_1)$, one adds the generalized half-plane $\CH^+(u)$ (defined 
to be the structure obtained by adding on the wedge $\CW^+(u)$ the generalized disks from $\CH$ that are glued to it). 

In this way, one obtains a loop-tree structure of LQG surfaces, that we call $\CH(u)$.  
Given the definition, it is clear that when $u' > u$, then $\CH(u')$ is (in some natural appropriate sense) ``embedded'' in $\CH(u)$, that $\CH(u)$ can be viewed as a Markov process. Its evolution can be described  as follows: 

We note that when $u$ increases, then $\CH(u)$ will make a ``jump'' at time $v$ when one of the following three possibilities occur (see Figures \ref {f6} and \ref {f7}): 
\begin{enumerate}[(i)]
\item If $x(v)$ is a boundary point of $\CW(0)$ where a generalized disk of $\CH$ was attached (in the foresting operation that constructed $\CH$ out of  $\CW(0)$), then this generalized disk ``disappears'' from $\CH(u)$ at time $v$. 
\item When $x(v)$ is the right endpoint of some $\CD(u')$ for $u' < v$ or is a point of an already discovered $\CLE_{\kappa'}$ loop that is isolated from the left on $\R_+$ (so that $x(v)$ is the endpoint of a bead of $\CW(u'')$ for $u'' < v$), then a loop-tree of LQG surfaces disappears from $\CH(u)$ at time $v$. This tree now  lies  ``to the left'' of $x(v)$ in $\CH(v-)$. 
\item When $x(v)$ is the first encountered boundary point of a $\CLE_{\kappa'}$ loop. In that case, one removes the interior of that loop from $\CD_0(u)$. 
\end{enumerate}
We see that in all three cases, the jumps correspond to some loop-trees of LQG surfaces: The ones that disappear as in (i) and (ii), and the ones that correspond to the interior of the CLE$_{\kappa'}$ loops that one removes. They also all come marked with the boundary point $x(v)$.  
For each time $u$, we denote by $\CF_u$ the $\sigma$-algebra generated by these marked loop-tree structures of LQG surfaces (mind that we do not record how they are embedded in the plane) up to time  $u$.

\begin{figure}[ht!]
\includegraphics[scale=.7]{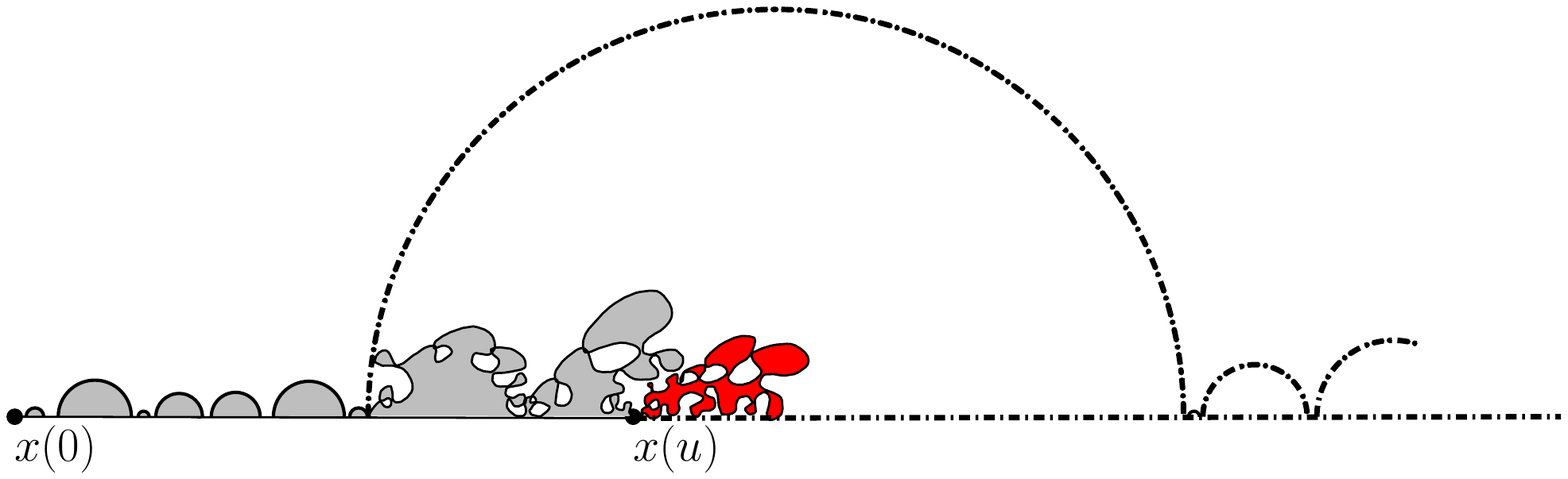}    
	\caption{Positive jumps of the boundary length: Discovering a new boundary-touching CLE loop (darker red shaded) creates an additional boundary length  (the dash-dotted boundary pieces are to-be-forested to obtain ${\mathcal H}(u)$).}
	\label{f6}
\end{figure}

\begin{figure}[ht!]
\includegraphics[scale=.7]{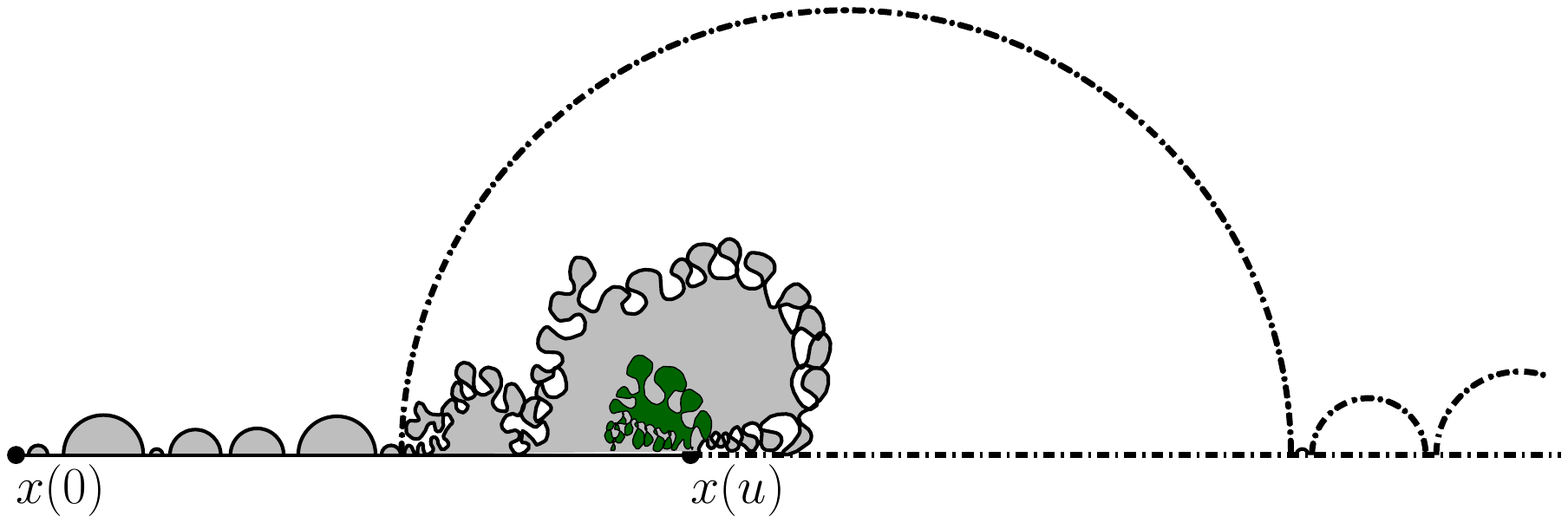}  
\includegraphics[scale=.7]{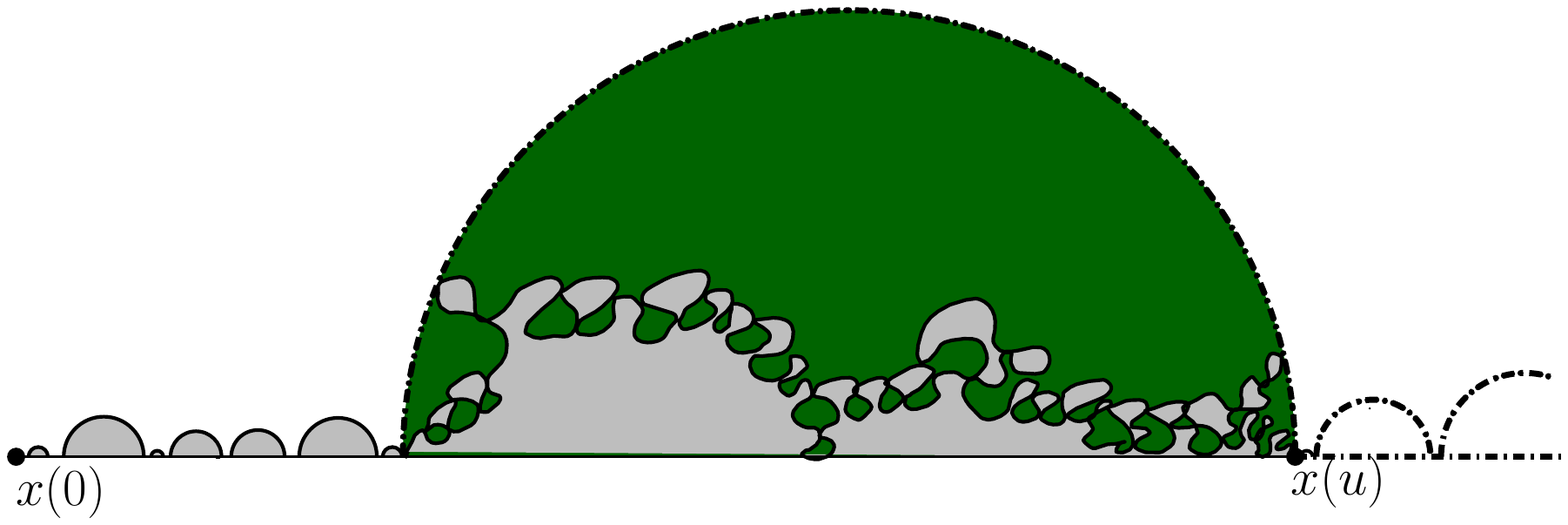} 
	\caption{Negative jumps for the boundary length. {\bf Top:} The generic case, where the darker green shaded region disappears from ${\mathcal H}(u)$, when $x(u)$ is the ``endpoint'' of a bead of $\CW (v)$ for some $v \in (0,u)$ (the dash-dotted boundary pieces are to-be-forested to obtain ${\mathcal H}(u)$).  {\bf Bottom:} The darker green shaded region disappearing from ${\mathcal H}(u)$ when $x(u)$ is the rightmost point of a bead of $\CW(0)$.}
	\label{f7}
\end{figure}

The first key proposition can now be stated as follows: 

\begin {prop} 
\label {main1}
For each $u>0$, $\CH(u)$ (with marked points $x(u)$ and $ \infty$) is a generalized quantum half-plane that is independent of $\CF_u$. 
\end {prop}

Let us make a few comments: 
\begin{enumerate}[(i)]
\item Proposition~\ref{main1} indeed implies Proposition~\ref{firstprop}. 
\item The  generalized disks that are glued to the boundary $[x(u), \infty)$ of $\CW(u)$  are anyway independent of the rest of the construction, so we do not really need to bother about those as they will clearly be a Poisson point process of generalized disks.   
\item Some of the generalized disks glued to the other side of the wedge $\CW(u)$ will be coming from those already present in $\CH$, some will be due to fjords created by discovered $\CLE_{\kappa'}$ loops, and some will correspond to the concatenation of a fjord-tree created by a $\CLE_{\kappa'}$ loop within $\CW(0)$ with  generalized disks already grafted to $\CW(0)$ (to form $\CH$).
\end{enumerate}

Before moving to the proof of Proposition~\ref{main1} in the next section, let us state and prove a first consequence.  Suppose that $v$ is some very large fixed constant. When $u \in [0, v]$, we can consider the {\em generalized} boundary length $l_u (x(u), x(v))$ of the counterclockwise boundary of $\CH(u)$ between $x(u)$ and $x(v)$. We can note that the fluctuation process $R_u^v:= u \mapsto l_u (x(u), x(v)) - l_0 (x(u), x(v))$ does in fact not depend on $v$, in the sense that for all $u \le v \le v'$, $R_u^v = R_u^{v'}$. This therefore defines a process $R_u$ for all $u \ge 0$. We note that this process will make a negative jump at each $u_0$ such that there is a generalized disk glued to $x(u_0)$ in $\CH$ (and in fact, we will explain in the next paragraph that $-R$ is a stable subordinator). 

Similarly, one can define the fluctuation of the generalized length $L$ of the clockwise boundary of $\CH(u)$ starting from $x(u)$. 
This process will have negative jumps (for instance at the endpoints of the beads of $\CW$) just like $R$, but it also has positive jumps (when a $\CLE_{\kappa'}$ loop is being discovered for the first time, then $L$ will have a positive jump given by the generalized boundary length of this loop).  

\begin {coro} 
The process $(-R_u)_{u \ge 0}$ is an $\alpha'$-stable subordinator, the process $(L_u)_{u \ge 0}$ is an $\alpha'$-stable process, and these two processes are independent.  
\end {coro}

\begin {proof}
Proposition~\ref{main1} shows immediately that the processes $R$ and $L$ are both L\'evy processes, and by construction $R$ has no positive jumps. Furthermore, $R$ is a function of the Poisson point process of disks attached to $[0, \infty)$ so that it is clearly independent of $L$ (that is a function of the $\CLE_{\kappa'}$ in the weight $W_D = \gamma^2 -2 $ wedge and of the disks attached to the other side of its boundary). The scaling properties of generalized quantum length in the generalized quantum disks (recall Remark~\ref{areascaling}) then implies that $L$ and $R$ are in fact both stable processes with index $\alpha'$ (recall that $\alpha'  = 4 / \kappa' \in (1/2, 1)$).
\end {proof} 

Let us now also explain the type of arguments that then allows us to describe the relative intensities of positive and negative jumps of the stable process $L$ (we will use the same ideas in the general case $p \in [0,1]$).

\begin {prop} 
The ratio between the intensity of positive and negative jumps of $L$ is $-2 \cos (\pi  \alpha' )$.  
\end {prop}

\begin {proof} 
A first observation is that by construction, the times $u$ at which $L_u$ attains its running infimum correspond exactly to the times in $\CC$, where the process is ``in between beads'' of the wedge $\CW$. Indeed, in all other cases, the process $L(u) \ge L(u_0)$ where $u_0$ is the left extremity of $\partial_u$. 

Now, it is a known feature of stable processes that the range of values of $-L(u)$ such that $L(u) = \min \{ L(v), v \le u \}$ is a stable subordinator of a certain index $1/\alpha''$ that can be expressed explicitly in terms of $\alpha'$ and of the ratio $U_L$ between the intensities of positive and negative jumps of $L$. More specifically (see \cite[Chapter~VIII, Lemma 1]{bertoin96levy}, where $\alpha''$ is equal to $\alpha'$ times the so-called positivity parameter of the process), one has the relation  
$$ U_L = \sin (\pi (\alpha' - \alpha'')) / \sin ( \pi \alpha'').$$ 
On the other hand, the negative jumps of $L^{\#}$ will have the same scaling properties as the generalized boundary lengths of the forested beads of $\CW$ which is given by the final expression in Remark~\ref{scalingremark}. We therefore get that $1+\alpha'' = 2 - \alpha'$. Plugging this into the previous expression for $U_L$ gives the result. 
\end {proof}

\begin {rema} 
It would be possible to try to show directly at this point that the rate of negative jumps of $L$ and the rate of negative jumps of $R$ actually coincide (which in turns determines the law of the pair $(R,L)$ up to a multiplicative constant), but we will derive this fact later in the general case. 
\end {rema}

\subsection{CLE exploration tree and proof of Proposition~\ref{main1}}
\label{subsec:ig}

The definition (and conformal invariance) of the conformal loop-ensembles $\CLE_{\kappa'}$ for $\kappa' \in (4,8)$ is based on the reversibility properties of $\SLE_{\kappa'}(\kappa'-6)$ processes, that have been derived using the ``imaginary geometry'' couplings of SLE-type processes with the GFF in \cite{MS_IMAG3,MS_IMAG4}. The properties that we will now recall and use can be viewed either as a consequence of the existence and properties of $\CLE_{\kappa'}$ or of the imaginary geometry coupling. We will first give the construction without reference to imaginary geometry, and then explain (in Remark~\ref{Rem:IG}) how this can be described in this framework. 

Let us recall how to define the collection of boundary-touching $\CLE_{\kappa'}$ loops in a simply connected domain with a marked boundary point via the corresponding SLE branching tree. 
For convenience, we first choose this domain to be the upper half-plane with marked point at infinity. 
Using their target-invariance property, one can define a branching-tree of $\SLE_{\kappa'}(\kappa'-6)$ processes starting from infinity, aiming at all boundary points $x \in \R$ (with marked point ``immediately to the left of infinity'' i.e. at $+\infty$ on the real line). In this way, for each $x \in \R$, $\eta_x'$ is an $\SLE_{\kappa'}( \kappa'-6)$ process from $\infty$ to $x$. And when $x \not= y$, the two processes $\eta_x'$ and $\eta_y'$ coincide until the first time at which they disconnect $x$ from $y$, and after this time, they evolve independently towards their respective target points.

{One can construct the collection of boundary touching loops out of this tree of processes $\eta_x'$ as follows. The idea is that in the end, for each given $x \in \R$, $\eta_x'$ (viewed as going from $x$ to $\infty$) will be the right side of the union of all $\CLE_{\kappa'}$ loops that touch the half-line $(- \infty, x]$. Let us first consider $\eta_0'$ (from $\infty$ to $0$).   Each excursion that $\eta_0'$ makes from $\R_-$ will then correspond to part of a boundary touching loop.  Suppose that we have such an excursion $\eta_0'|_{[s,t]}$.  Then $\eta_0'(s),\eta_0'(t) \in \R_-$ with $\eta_0'(s) < \eta_0'(t)$. If we condition on $\eta_0'$,  the $\CLE_{\kappa'}$ loop containing $\eta_0'|_{[s,t]}$ is then completed by concatenating $\eta_0'|_{[s,t]}$ with the part of $\eta_{y+}'$ for $y = \eta_0'(s)$ which is in the component of $\h \setminus \eta_0'([s,t])$ that has $[y, \eta_0' (t)]$ on its boundary (more precisely, if we have fixed a countable dense set $(x_n)$ of $\R$ and $(x_{n_j})$ is a subsequence of $(x_n)$ which decreases to $y$, then the rest of the loop is given by the limit as $n \to \infty$ of the part of $\eta_{x_n}'$ which is in the aforementioned component).  This gives the boundary intersecting loops of the $\CLE_{\kappa'}$ part of which are drawn by $\eta_0'$.  By considering the same construction with all of the $\eta_{x_n}'$ we can construct all of the boundary touching loops.}

\begin {rema} 
\label {remBCLE}
 The $\BCLE_{\kappa'}(\rho')$ processes defined in \cite {cle_percolations} are constructed in exactly the same way, except that one considers the branching tree of $\SLE_{\kappa'} (\kappa' - 6- \rho';\rho')$ processes (for $\rho' \not= 0 $ this time) instead of $\SLE_{\kappa'}(\kappa' -6)$
 (these processes are also target-independent, so that it is possible to construct such a branching tree).  We will use these $\BCLE_{\kappa'} (\rho')$ processes in the study of the general case $p \in [0,1]$, where they show up naturally.
 \end {rema} 

 \begin {rema} 
 \label {Rem:IG}
 As mentioned above, one way to understand (and to actually prove some of its features) the above constructions is to use the imaginary geometry framework, and to construct all these loops out of a GFF. For the $\CLE_{\kappa'}$, one can start with a GFF $h^{\IG}$ on $\h$ with boundary conditions given by $\lambda'-\pi \chi$ on the real line, where $\lambda' = \pi/\sqrt{\kappa'}$ and $\chi = 2/\sqrt{\kappa}-\sqrt{\kappa}/2$. Then, for each $x$, one defines $\eta_x'$ to be the counterflow line of $h$ from $\infty$ to $x$; it turns out (using the results of \cite{MS_IMAG}), that the joint law of these curves have all the properties described above.  For the $\BCLE_{\kappa'}(\rho')$ variant, one uses the same construction but with a GFF $h^\IG$ with boundary conditions $\lambda'(1+\rho')-\pi \chi$ on $\R$. This type of GFF-based construction was one of the starting points of \cite {cle_percolations}.   
\end {rema}

\begin{figure}[pt!]
\includegraphics[width=10cm]{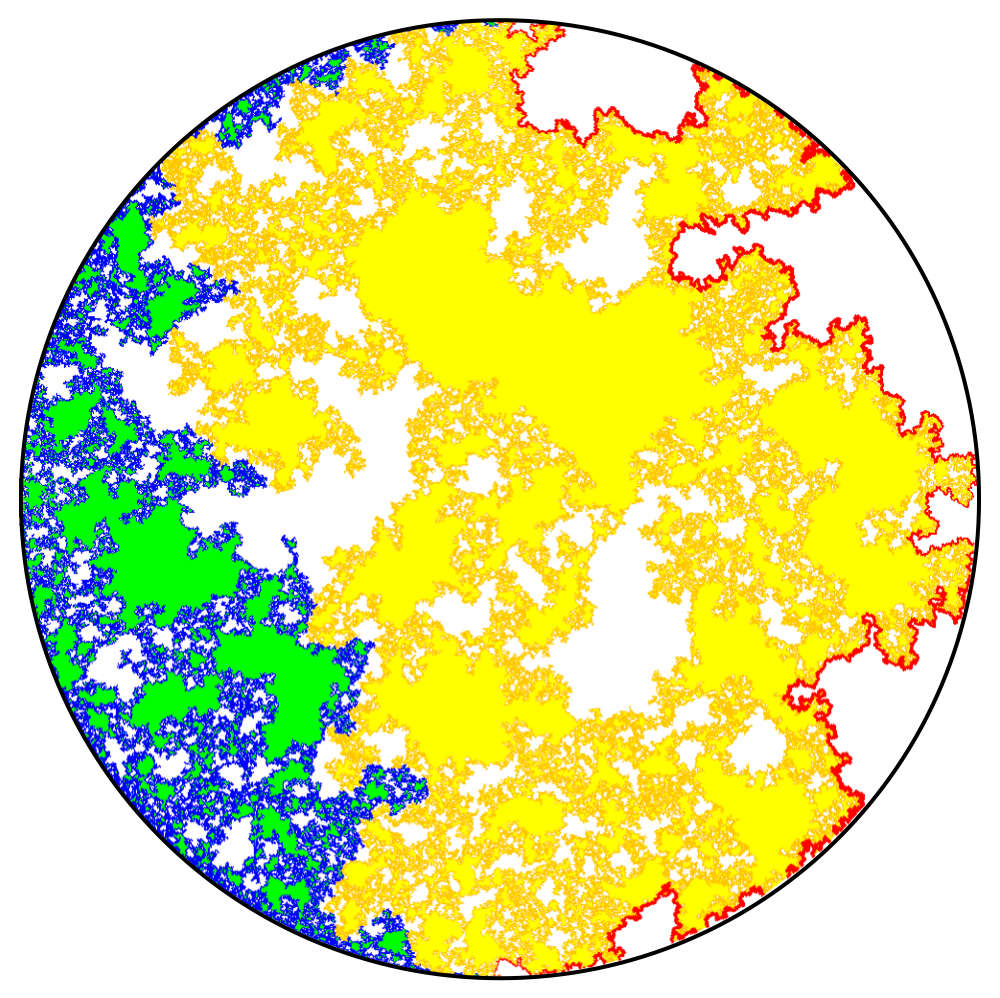}  
	\caption{A simulation (mapped onto the unit disk, and for $\kappa'=6$) of the situation at time $0$: In blue and yellow, are the   $\CLE_{\kappa'}$ loops that touch the left-hand semi-circle (corresponding to the image of the negative half-line). The boundary of the connected component of the complement of these loops that touches the right-hand semi-circle is drawn in red. This is a simple curve from $-i$ to $i$. The pieces to the right of this red curve form the beads of the wedge $\CW(0)$. The $\CLE_{\kappa'}$ loops that touch the left half-circle are drawn in blue or yellow depending if they contribute to the red curve of not. The right-hand part of the fjord structure in the yellow loops corresponds to the foresting of the left boundary of the wedge $\CW(0)$ that give rise to $\CH$ (once one also adds the foresting on the right half-circle).}
	
	\label {f8}
\end{figure}

Let us now condition on $\eta_0'$, and consider the connected components of the complement of $\h \setminus \eta_0'$ that touch the positive real half-line (i.e., that contain sub-intervals of the positive real half-line). In the previous construction of the boundary-touching CLE$_{\kappa'}$ via the tree of $\eta_x'$ processes, the path $\eta_0'$ traces the ``outer right'' part of the union of the clusters that touch the negative half-line. Hence,  conditionally on $\eta_0'$, the remaining parts of the CLE$_{\kappa'}$ in the connected components that lie between $\eta_0'$ and the positive half-line consist of independent CLE$_{\kappa'}$ processes in those domains. 
It follows that, conditionally on $\eta_0'$, the remaining parts of the paths 
$\eta_x'$ for $x > 0$ after they split from $\eta_0'$ will trace the loops of a CLE$_{\kappa'}$ that touch $[0, \infty)$ in these domains.  

Let us now combine this exploration-tree setup with the LQG surface setup, i.e., we endow $\h$ with an independent quantum half-plane structure $\CW = (\h,h,0,\infty)$. We will also forest this half-plane using a Poisson point process of generalized disks on its boundary.  

A first main observation is that (see the remark after Theorem~\ref{thm:gluing_wedges'})  the quantum surface  corresponding to the domain that lies to the right of $\eta_0'$ is a generalized half-plane -- we denote it by $\CH$ (the marked point being $0$), see Figure \ref {f8}.

When $u > 0$, we define $x(u)$ in $\R_+$ to be so that the quantum length of $[0,x(u)]$ is equal to $u$.  By translation invariance of the half-plane, $(\h, h, x(u), \infty)$ is also a quantum half-plane, and consequently, the quantum surface $\CH(u)$ corresponding to the domain that lies to the right of $\eta_{x(u)}'$  is also a  generalized quantum half-plane that we will denote by $\CH(u)$, see Figure \ref {f9}.

\begin{figure}[ht!]
\includegraphics[width=10cm]{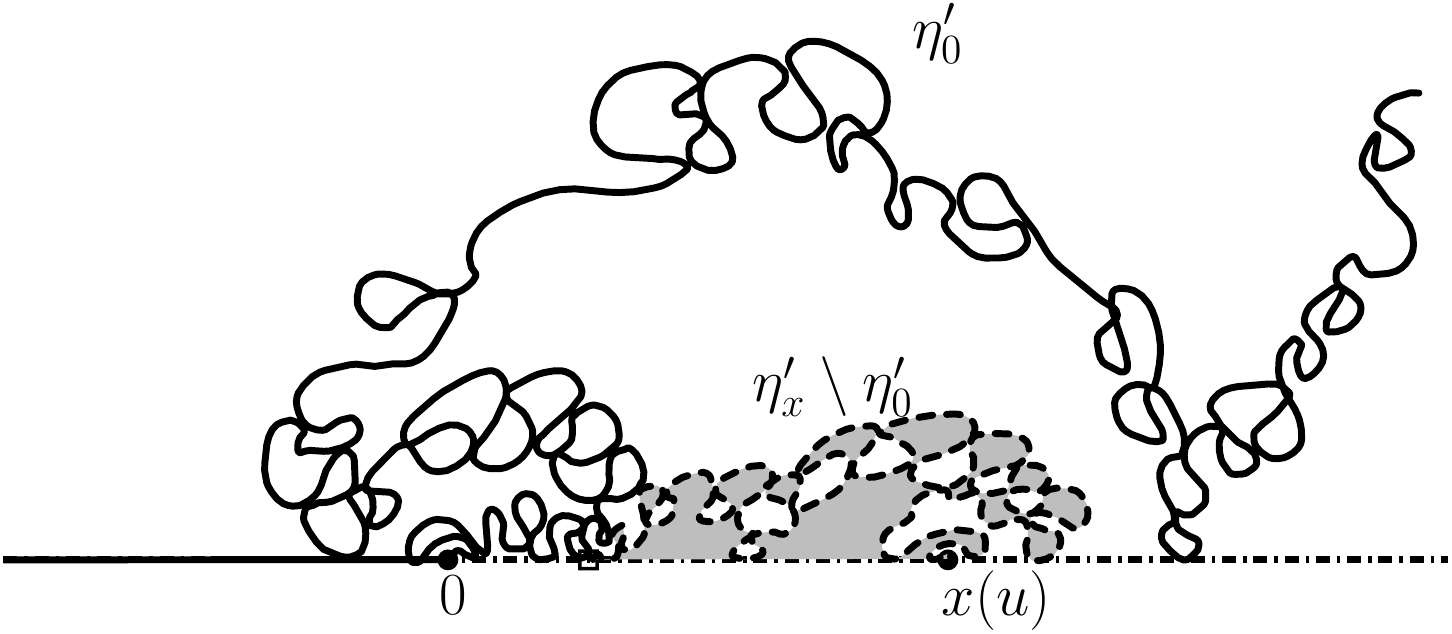}   
	\caption{The pieces underneath the right of $\eta_0'$ form $\CH$ (the dashed-dotted half-line being forested) and similarly, the pieces underneath the right of $\eta_{x(u)}'$ form $\CH(u)$. One can view $\eta_x' \setminus \eta_0'$ (in dashed) as the upper boundary of the $\CLE_{\kappa'}$ loops in $\CH(0)$ that touch $[0, x]$.}
	
	\label {f9}
\end{figure}

Let us condition on $\eta_0'$. As explained above, the remainder of all the $\eta_x'$ for $x >0$ can be interpreted as the discovery of all loops that touch $\R_+$ in a $\CLE_{\kappa'}$ in the connected components of the wedge of weight $W_D=\gamma^2-2$ that forms the spine of $\CH(0)$.

We have therefore a simple way of defining all the structures that appear in Proposition~\ref{main1}: The forested quantum half-plane 
together with $\eta_0'$ defines the generalized half-plane $\CH$ and the remainder of all the $\eta_x'$ for $x >0$ then defines the $\CLE_{\kappa'}$ loops that are used to define $\CH(u)$. In particular, we can note that $\eta_{x(u)}' \setminus \eta_0'$ traces exactly the (generalized) boundary of ${\mathcal D}_0(u)$ that is not part of the boundary of ${\mathcal D}(u)$. We can then conclude that the generalized LQG surface $\CH(u)$ in Proposition \ref {main1} is indeed a generalized quantum half-plane. 

To complete the proof, it then suffices to note that the $\sigma$-algebra $\CF_u$ is independent of $\CW(u)$, because the (forested) quantum wedge that lies to ``the other side'' of $\eta_{x(u)}'$ is independent from $\CH(u)$, and that the random surfaces cut out before time $u$ can be constructed by sampling appropriate and independent SLE-type curves within this quantum surface.

\subsection {Jumps correspond to generalized quantum disks}
Recall that we have noted that one can associate to each jump of the process ${\mathcal H}(u)$ (i.e., to each jump of $R$ and $L$) a loop-tree structure of LQG surfaces.

\begin{prop}
 \label{main2} 
 In the setup of Proposition~\ref{main1}, one has three independent Poisson point processes of generalized disks (corresponding respectively to the positive jumps of $L$, the negative jumps of $L$ and the negative jumps of $R$). 
\end{prop}
\begin {rema} 
We are now going to use some of the ideas that were already instrumental in \cite {cle_percolations}. 
One consequence of the imaginary geometry setup described in Remark~\ref{Rem:IG} is that it allows us to describe also the joint distribution of all the paths $\eta_x'$ with their outer boundaries. If $\eta_L$ and $\eta_R$ respectively denote the flow lines of the GFF $h^{\IG}$ starting at $0$ with respective angles $\pi/2$ and $-\pi/2$, then these paths will be the outer-left and outer-right boundaries of $\eta_0'$. Furthermore: (a) the law of $\eta_L$ and $\eta_R$ is that of an $\SLE_\kappa(-\kappa/2;\kappa/2-2)$ and an $\SLE_\kappa(2-\kappa;\kappa-4)$ process in $\h$ from $0$ to $\infty$. (b) The conditional law of $\eta_0'$ given $\eta_L$ is {that of} an $\SLE_{\kappa'}(\kappa'/2-4)$ process in the domain to the right of $\eta_L$, the conditional law of $\eta_0'$ given $\eta_R$ is that of an $\SLE_{\kappa'}(\kappa'-6;\kappa'/2-4)$ process in the domain to the left of $\eta_R$. (c) The conditional law of $\eta_0'$ given both $\eta_L$ and $\eta_R$ is that of an $\SLE_{\kappa'}(\kappa'/2-4;\kappa'/2-4)$ process in each of the domains between $\eta_R$ and $\eta_L$. 
\end {rema}

\begin{proof}
Let us first note that the fact that the downward jumps of $R$ correspond to a Poisson point process of generalized quantum disks that is independent of the jumps of $L$ (and the corresponding surfaces) is a direct consequence of the construction: These are the generalized disks that had been forested to the positive half line of $\CW(0)$, and that are indeed independent of the rest. It therefore remains to look at the surfaces corresponding to the jumps of $L$. 

In order to understand the intensity measures of the Poisson point processes of surfaces that are cut out of the generalized half-planes ${\mathcal H}(u)$, we can focus on the law of these cut-out surfaces during the time-interval $[0,u]$ and take $u \to 0$.  

Let us now work in the same LQG setup as in the proof of Proposition~\ref{main1}.  We first look at the positive jumps of $L$, i.e., that correspond to the discovery of a boundary-touching $\CLE_{\kappa'}$ loop. 
Fix $\delta$ small and positive (we will eventually let $\delta \to 0$ as well).  We first fix $u >0$ and let $x=x(u)$.  Let $E=E_{u}$ be the event that there exists a boundary-touching $\CLE_{\kappa'}$ loop ${\mathcal L}$ that disconnects $x$ from infinity, such that the left-most point of $I_{\mathcal L}:= {\mathcal L} \cap \R$ is in $[0, x(u)]$, that the outer part $o({\mathcal L})$ of the loop ${\mathcal L}$ (going clockwise from the left-most point of $I_{\mathcal L}$ to its rightmost point) is traced by $\eta_x'$ and has quantum length at least $\delta$. 
Scaling and root-invariance of the quantum half-plane shows that the probability of $E_u$ decays like $c(\delta) \times u$ as $u \to 0$.  

The following observation will also be useful: Suppose that $E_u'$ denotes the same event as $E_u$, except that we do not impose the condition that the loop is traced by $\eta_x'$ (this means that we allow the possibility that this loop is ``hidden'' underneath $\eta_x'$). Then $\p[ E_u \giv E_u'] \to 1$ as $u \to 0$. Indeed, if it was not the case, then Proposition~\ref{main1} would imply that with positive probability, two macroscopic $\CLE_{\kappa'}$ loops happen to have the same left-most boundary point, which we know can not happen. 

Let $\eta_{x,L}$ be the left boundary of $\eta_{x}'$ (viewed as a path from $x$ to $\infty$). We now have the following features: (a) By Theorem~\ref{thm:gluing_wedges}, the domain to the right of $\eta_{x,L}$ (i.e., where $\eta_x'$ is) is a quantum wedge ${\mathcal W}_1$ of weight $\gamma^2 / 2$ and the domain to its left is an independent wedge ${\mathcal W}_2$ of weight $2 - \gamma^2 /2$. (b) The conditional law of $\eta_x'$ given $\eta_{x,L}$ is that of an $\SLE_{\kappa'} (\kappa' /2 - 4)$ in this first wedge ${\mathcal W}_1$. In particular, the excursions away from $\eta_{x,L}$ traced by $\eta_x'$ within ${\mathcal W}_1$ form a Poisson point process of generalized quantum disks that is independent of ${\mathcal W}_2$. 

Each excursion of $\eta_x'$ away from $(- \infty, x]$ traces the outer piece of a $\CLE_{\kappa'}$ loop, and it is obtained by concatenating all the excursions that $\eta_x'$ makes away from an excursion of $\eta_{x,L}$ away from $(- \infty, x]$, see Figure \ref {f10}. Note that these excursions of $\eta_{x,L}$ correspond to the beads of ${\mathcal W}_2$. To complete that $\CLE_{\kappa'}$ loop, one has to draw the little missing piece of the loop, which lies below $\eta_{x,L}$, i.e, within a bead of ${\mathcal W}_2$ -- see Figure \ref {f11}.

\begin{figure}[ht!]
\includegraphics[width=10cm]{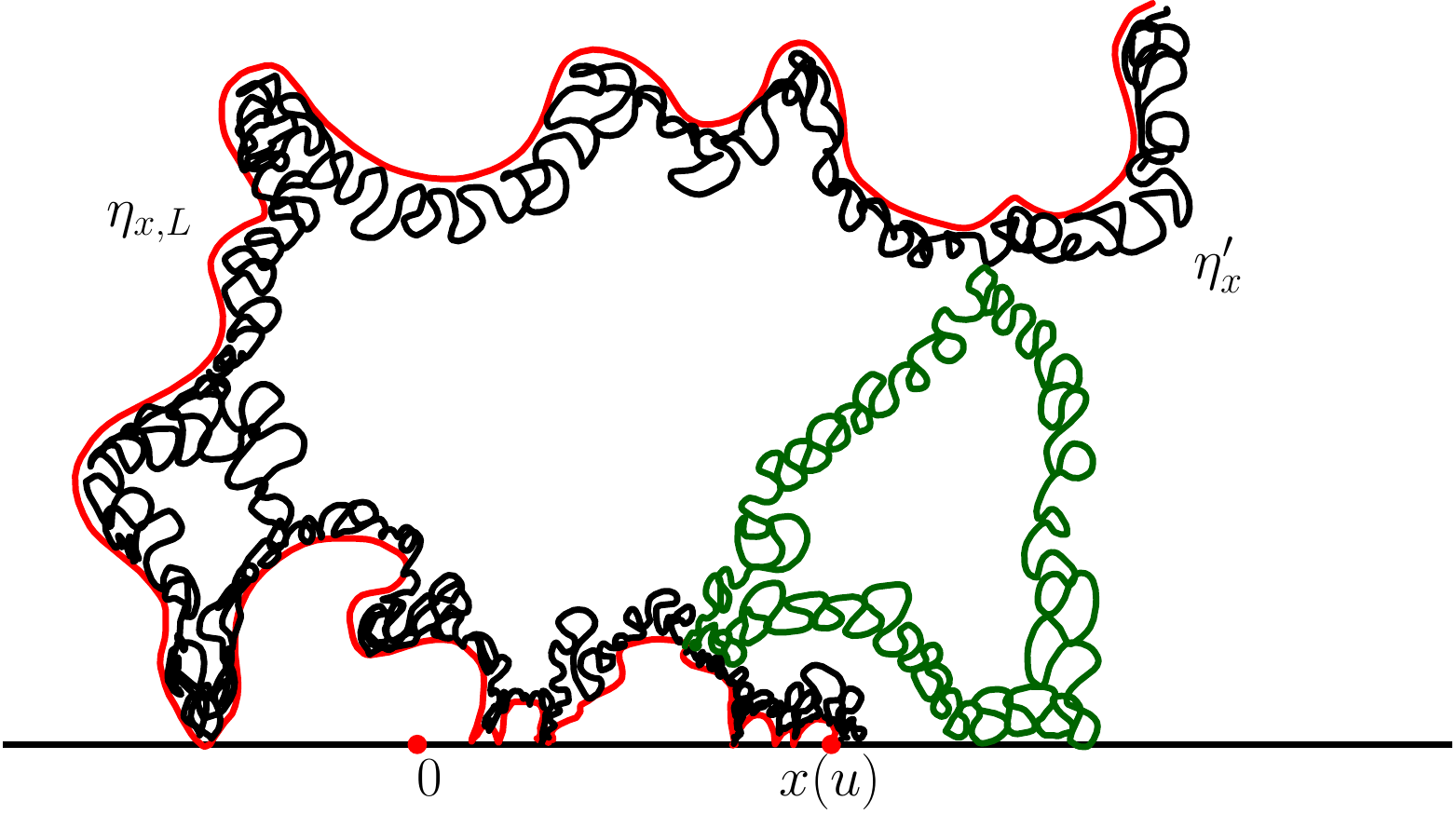}   
	\caption{The curve $\eta_{x,L}$ and $\eta_{x}'$ when $E_u$ occurs (the large excursion of $\eta_x'$ away from $\eta_{x,L}$ in dark green).}
	\label{f10}
	
\end{figure}

\begin{figure}[ht!]
\includegraphics[width=10cm]{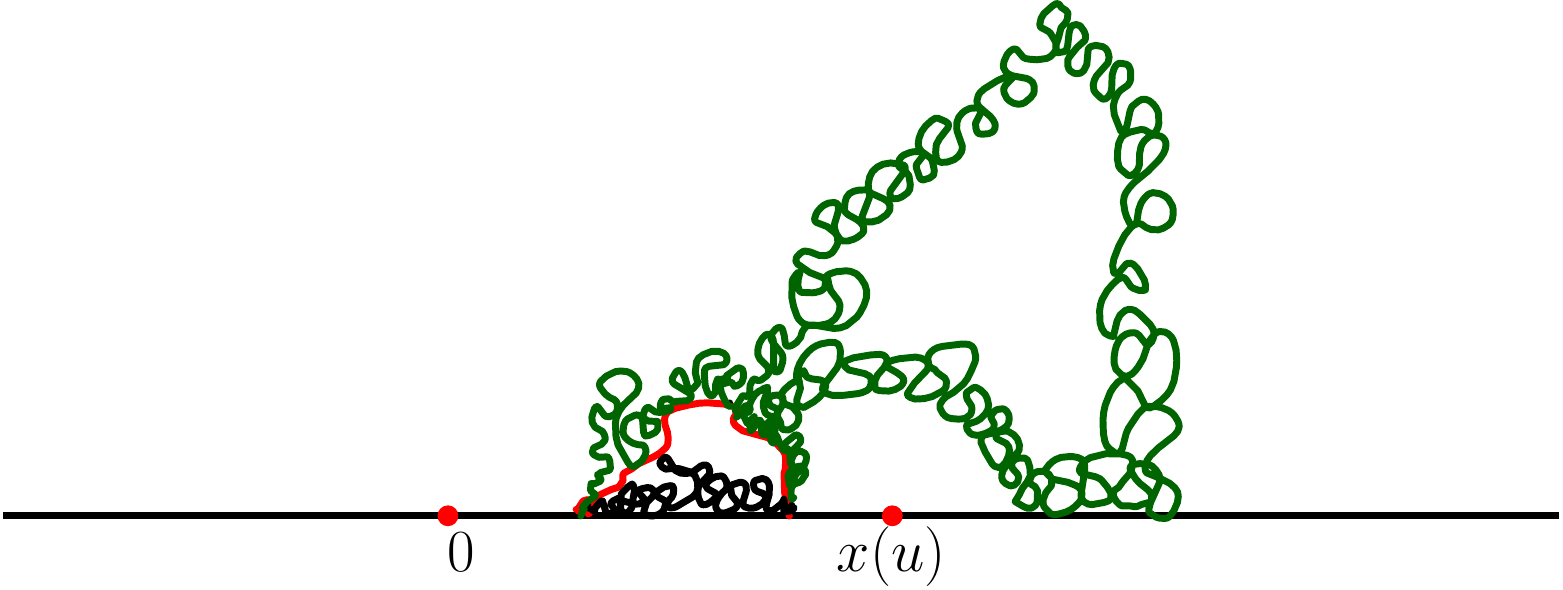}  
	\caption{Completing the large $\CLE_{\kappa'}$ loop that intersects $[0, x(u)]$}
	\label {f11}
	
\end{figure}

When $u$ is very small, one can wonder whether $\eta_{x, L}$ will do something exceptional when one conditions on $E_u$. Simple scaling considerations show that it does not, i.e, that the conditional boundary lengths of the beads of $\eta_{x,L}$ containing some piece of $[0, x(u)]$ on their boundary remains comparable to $u$ (the ratio remains tight). In particular, the quantum length of the piece of the loop under $\eta_{x,L}$ will be typically very small and the area of the entire bead will be very small as well.   

On the other hand, conditionally on $E_u$, $\eta_x'$ will typically make one very large excursion away from $\eta_{x,L}$ (the interior of which is a generalized disk), while all other excursions will be very small as well (this is simply due to the standard fact that conditioning a subordinator to take a very large value at a given time is essentially equivalent to conditioning on the existence of a very large jump). 

So, for very small $u$, we see that conditionally on $E_u$, the inside of the discovered large $\CLE_{\kappa'}$ loop will consist of the generalized quantum disk of boundary length greater than $\delta$, to which one attaches a quantum surface of very small quantum area at a boundary-typical point.  Letting $u \to 0$, one readily concludes that conditionally on $E_u$, the law of the cut-out surface is that of a generalized quantum disk with boundary length at least $\delta$ (and renormalized to be a probability measure). 

Let us now look at the downward jumps of $L$. This time, we use $\eta_0'$ and its right-boundary $\eta_{0,R}$. We note that when $u$ is very small, $L$ will make a large negative jump before time $u$ when $\eta'$ makes a large excursion away from some point in $[0, x(u)]$ to some other point. This excursion will then form the ``upper boundary'' of the cut out surface, and the lower boundary will then be completed by an SLE-type process in the corresponding bead under $\eta_{0,R}$. The above arguments for the positive jumps can be readily adapted to see that those cut out surfaces are also a point process of generalized quantum disks.
\end{proof}

\section{Results for general $p$} 
\label{S4}
\subsection{Main statement}

We now explain the results for general $p \in [0,1]$. Suppose that we have the same setup as in Propositions~\ref{main1} and~\ref{main2}: We start with a generalized quantum half-plane $\CH$ that is obtained by adding the Poisson point process of generalized disks to the boundary of a weight $W_D = \gamma^2 - 2$ wedge $\CW$. 
We consider a $\CLE_{\kappa'}$ in $\CW$ (i.e., an independent one in each of the quantum disks that form $\CW$) and we this time color each $\CLE_{\kappa'}$ loop red (resp.\ blue) independently with probability $p$ (resp.\ $1-p$). We then define in $\CW$ the interface $\eta$ between the red clusters that touch the clockwise boundary arc of $\CW$ from $x_0$ to infinity, and the blue clusters that touch the counterclockwise boundary arc.  

By \cite{cle_percolations}, we know that for each $p \in [0,1]$ there exists $\rho \in [-2,\kappa-4]$ so that this interface is an $\SLE_{\kappa} (\rho; \kappa-6-\rho)$ process. More precisely, this path is the union/concatenation of the corresponding interfaces in each of the quantum disks of $\CW$ (from one of the marked points to the other). When $p=0$ (resp.\ $p=1$), this interface obviously follows the left (resp.\ right) boundary respectively and this corresponds to $\rho=-2$ (resp.\ $\rho=\kappa-4$). By symmetry, we also know that when $p=1/2$, $\rho = (\kappa-6)/2$.

The continuous path $\eta'$ obtained by following the interface $\eta$ and tracing each encountered $\CLE_{\kappa'}$ loop (clockwise or counterclockwise, depending on its color) at the first time at which the trunk encounters it, is called the {\em full} $\SLE_{\kappa'}^\beta$ (or full $\SLE_{\kappa'}^\beta (\kappa'-6)$)  for $\beta = 2p-1 \in [-1, 1]$ in \cite{cle_percolations}, but we will not really use this terminology here, as in the present paper, we will rather want to interpret the discovery of $\CLE_{\kappa'}$ loops as jumps. 

We now choose to parameterize the interface $\eta$ according to its quantum length. Let us now explain how to define the quantum surfaces $\CH(t)$ and $\CW(t)$ (we opt for a somewhat heuristic description of the definition here, that put together with the figures is hopefully more enlightening than the dry formal definition). One main difference with the $p=0$ and $p=1$ cases is that when $p \in (0,1)$, $\eta$ will touch both sides of the disks of $\CW$ that it traverses (when it goes from one marked boundary point to the other, it will hit both the clockwise and the counterclockwise boundary arcs joining these two points infinitely many times). 
We let $\CD(t)$ denote the bead of the initial wedge $\CW$ in which $\eta(t)$ is (for each given positive $t$, this is indeed almost surely well-defined).  
Our rules will be that whenever $\eta$ hits the boundary of the disks or hits already traced $\CLE_{\kappa'}$ loops, then it cuts away the disconnected pieces (that lie on the other side of the direction in which $\eta$ is then heading). These pieces then fall off from $\CH(t)$ and $\CW(t)$. On the other hand, when a newly discovered $\CLE_{\kappa'}$ loop appears, one removes only the inside of this loop, but keeps its ``fjords'' in $\CW(t)$ (they will be part of the generalized surfaces in $\CH(t)$). We note that at these times, $\eta(t)$ will in fact almost surely be at the end of a chain of such fjords (at these times, we also interpret $\eta(t)$ to be on the ``side'' that corresponds to the color of the discovered loop, so that $\eta$ will continue leaving the red loops on its left and the blue loops on its right, see Figure \ref {f12}).

\begin{figure}[ht!]
\includegraphics[width=6cm]{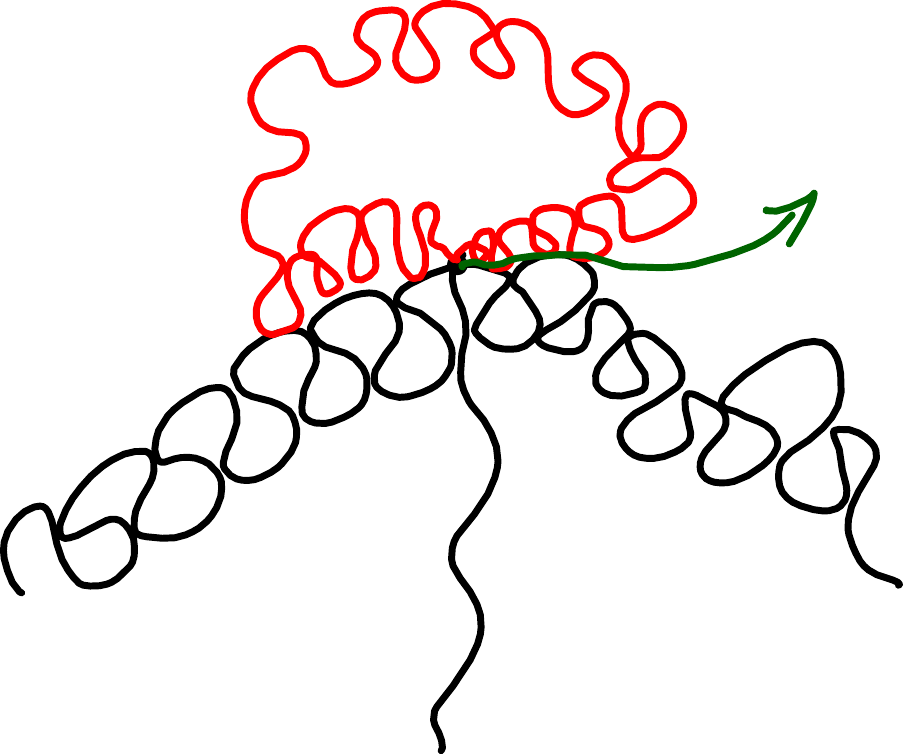}
\hspace{0.05\textwidth}
\includegraphics[width=6cm]{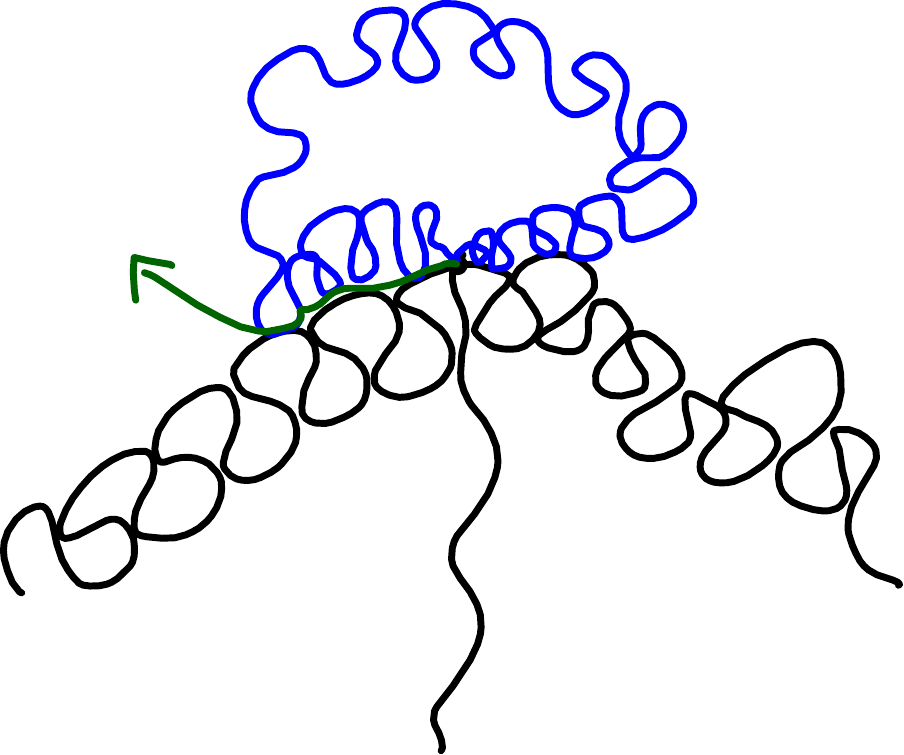} 
	\caption{Sketch: $\eta$ discovers a $\CLE_{\kappa'}$ loop at time $t$ (depicted is $\eta$ up to time $t$, a portion of the boundary of ${\mathcal H}(t^-)$ and the discovered loop). One chooses the ``side'' of the boundary point $\eta(t)$ according to the color of the discovered $\CLE_{\kappa'}$ loop.}
	\label {f12}
\end{figure}

In this way, one obtains for each time $t$, two quantum surfaces  $\CW(t)$ and $\CH(t)$ with one marked boundary point $\eta(t)$ (and the other boundary point at infinity), see Figure \ref {f13} for a sketch.

\begin{figure}[ht!]
\includegraphics[scale=.7]{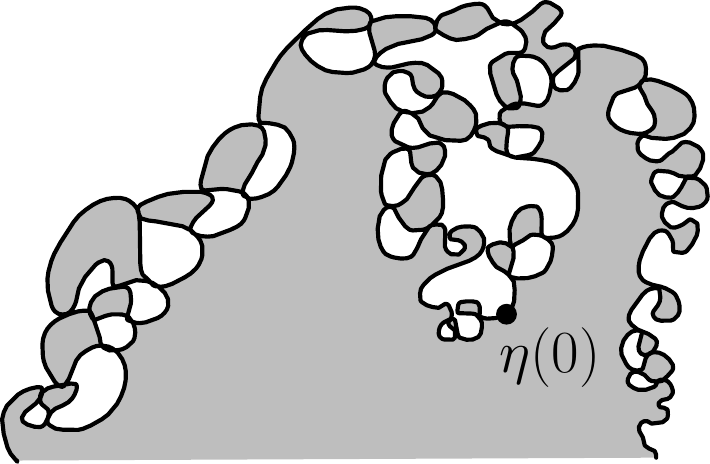}  
\quad
\includegraphics[scale=.7]{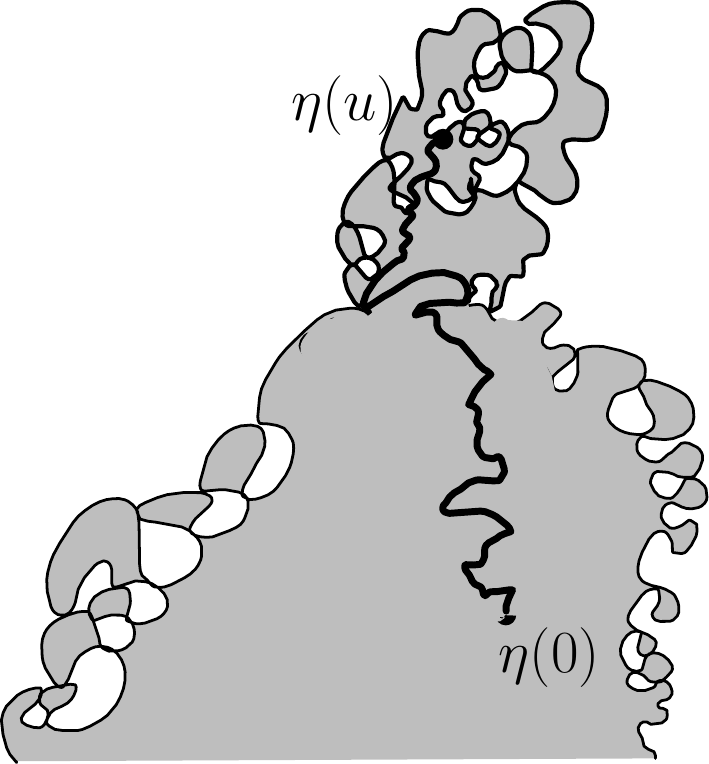}   
	\caption{Sketch of ${\mathcal H}(0)$ and ${\mathcal H}(t)$ (non-shaded regions)}
	\label {f13}
\end{figure}

We can then define the two processes $(R_t)_{t \ge 0}$ and $(L_t)_{t \ge 0}$ in a similar manner as for $p=1$. We note that this time, $R$ and $L$ will both have positive and negative jumps. The positive jumps of $R$ correspond to the times at which the interface discovers a new $\CLE_{\kappa'}$ loop that lies to its right (i.e., a blue loop), and the negative jumps correspond to times at which the interface disconnects some piece to its right (by either hitting the right boundary of the wedge, or an already discovered blue loop).  By symmetry, all the analogous statements hold for the jumps of $L$. This defines four point processes of quantum surfaces, corresponding respectively to the inside of the blue loops giving rise to a positive jump of $R$, the inside of the red loops giving rise to a positive jump of $L$, the surfaces that are cut out by the trunk to its right and that are cut out by the trunk to its left. All these surfaces have $\eta(t)$ as a marked boundary point. Finally, we let $\CF_t$ be the $\sigma$-algebra generated by these four point processes of quantum surfaces up to time $t$. 

We can now state what can be viewed as the main key result of the present paper (recall that $\kappa' \in (4,8)$ and $\alpha' = 4 / \kappa'$): 

\begin{theo}
\label{thm3} When $p \in (0,1)$, the following statements hold: 
\begin {enumerate}[(i)]
 \item\label{it:unexplored} For each $t \ge 0$, the quantum surface $\CH(t)$ is a generalized quantum half-plane that is independent of $\CF_t$. 
 \item\label{it:bl_proc} The two processes $L$ and $R$ are independent $\alpha'$-stable L\'evy processes. 
 \item\label{it:jumps_and_p} The ratio between the rates of positive jumps of $R$ and the rates of positive jumps of $L$ is $p / (1-p)$. 
 \item\label{it:jump_ratio} The ratio $U_L$ (resp.\ $U_R$) of the intensity of upward to downward jumps of $L$ (resp.\ $R$) is given by
\[ U_L = \frac{\sin(-\pi \rho/2)}{\sin(\pi \rho/2- \pi \alpha')}\quad\text{and}\quad U_R = \frac{\sin(2 \pi \alpha' - \pi \rho/2)}{\sin( \pi \rho / 2  - \pi \alpha')}.\] 
 \item\label{it:disks_are_disks} The jumps of $R$ and $L$ correspond to four independent Poisson point processes of quantum disks. 
\end {enumerate} 
\end{theo}

\begin {rema} 
The picture will be completed 
in Section~\ref{S51}, where it will be explained why the rate of negative jumps of $R$ and the rate of negative jumps of $L$ are the same. This will in particular show that $p / (1-p) = U_R / U_L$ which then gives the relation between $p$ and $\rho$ (i.e., Theorem~\ref{thm:betarhoformula}). 
\end {rema}

\begin {rema} 
We will state and prove the counterpart of this result for explorations of generalized quantum disks as Theorem~\ref{thm4} later (and the proof will build on the result in half-planes). This will be the result that then can be used to fully explore LQG surfaces via CLE exploration mechanisms. 
\end {rema}

\begin {rema} 
One approach to prove Theorem~\ref{thm3} is to build on Propositions~\ref{main1} and~\ref{main2} for $p=0$ and $p=1$, and to use the following approximation of the interface $\eta$ for general $p$:

Fix $\delta > 0$; toss an independent coin which is heads (resp.\ tails) with probability $p$ (resp.\ $1-p$) and depending on its outcome we let $\eta_\delta$ evolve along the left (resp.\ right) boundary, up until it has traced boundary quantum length $\delta$. We also attach to this path the $\CLE_{\kappa'}$ loops that it has encountered. Then, by Proposition~\ref{main1}, we still have to explore a quantum half-plane $\CH(u)$ (or rather, the wedge $\CW(u)$). 
We then toss a second independent $p$ v.s.\ $1-p$ coin to decide whether the second stretch of the exploration will follow the left or right boundary of the wedge $\CW(u)$. Note that if the outcome of the second coin disagrees with that of the first one, then this interface will start moving in the interior of the initial wedge $\CW$. 
We then continue iteratively, tossing a new independent coin at each time which is a multiple of $\delta$. 

Then, for each given $\eps$ and $K$, as $\delta$ gets smaller and smaller, the probability that during at least one of the first $K/\delta$ stretches, one discovers two $\CLE_{\kappa'}$ loops of diameter greater than $\eps$ goes to $0$. Together with the fact that the clusters of loops of diameter greater than $\eps$ converge to the clusters of loops (this is a non-trivial fact, shown in \cite{cle_percolations}), this shows that the interface $\eta_\delta$ indeed converges (in distribution) to the actual interface $\eta$ (it is actually also possible to couple $\eta_\delta$ with $\eta$ by deciding to color the largest $\CLE_{\kappa'}$ loop encountered by $\eta_{\delta}$ at the $k$th iteration (on $[k\delta, (k+1) \delta]$ quantum time) according to the coin-toss performed at time $k \delta$).

Hence, we see that this side-swapping interface $\eta_\delta$ indeed converges to the real interface $\eta$ as $\d \to 0$, and that the collection of discovered CLE loops converge as well. Given the results that we have derived for $p=0$ and for $p=1$, this suggests that the Poissonian structure of the appearing/disappearing surfaces should still hold, and that the only difference will lie in the fact that the positive jumps will be jumps of $L$ with probability $p$ and jumps of $R$ with probability $1-p$. 

However, to make this type of proof rigorous, some arguments are needed to justify the fact that the boundary lengths (of the cut-out domains for instance) of the approximations converge to those discovered by $\eta$. We will instead follow another route to prove this, similar to the one that we used to prove Propositions~\ref{main1} and~\ref{main2}. 
\end {rema}

\subsection {The BCLE decomposition}
We now briefly recall the BCLE description of the $\CLE_{\kappa'}$ from \cite {cle_percolations}:  
Suppose that $\kappa' \in (4,8)$ and  $\rho' \in (\kappa'/2-4,\kappa'/2-2)$. Recall from Remark~\ref{remBCLE} that a $\BCLE_{\kappa'} (\rho')$ is constructed via the branching tree of $\SLE_{\kappa'} (\rho'; \kappa'-6- \rho')$ processes. In the particular case where $\rho'=0$ (or symmetrically, when $\rho' = \kappa' -6$), this is exactly the picture one obtains when one discovers the boundary-touching $\CLE_{\kappa'}$ loops.

When $I$ is a subinterval of $\R$, we can define the collection of $\BCLE_{\kappa'}$ loops that do touch $I$. By conformal invariance, we can define also $\BCLE_{\kappa'}(\rho')$ in any simply connected domain in the plane.

Suppose that one draws a colored $\CLE_{\kappa'}$ in the upper half-plane and that one traces the entire interface $\eta$ from $0$ to $\infty$. As we have already mentioned several times, it is shown in \cite {cle_percolations} that the law of $\eta$ is that of an $\SLE_{\kappa} (\rho ; \kappa -6- \rho)$. A further result of \cite {cle_percolations} is the description of the conditional law of $\eta'$ given $\eta$ that we now recall. We call $H_L$ and $H_R$ the two domains that lie respectively to the left and to the right of $\eta$ (each of them is the union of disjoint simply connected domains). The $\CLE_{\kappa'}$ loops that touch $\eta$ from the left (i.e., that are in $H_L$) are by construction red, while the ones that touch $\eta$ from the right are blue.   

\begin {prop}[Theorem 7.2 from \cite {cle_percolations}] 
\label {BCLEinCLE} 
If one conditions on the whole of the interface $\eta$, then the collection of $\CLE_{\kappa'}$ loops that touch $\eta$ from the left is conditionally independent from the collection of $\CLE_{\kappa'}$ loops that touch $\eta$ from the right. Furthermore, the conditional law of the $\CLE_{\kappa'}$ loops that touch the right side of $\eta$ is obtained by taking independent $\BCLE_{\kappa'} (\rho_R')$'s in each connected component of $H_R$ for $\rho_R'= - \kappa' (\rho+2) / 4$ , and to keep only those loops that touch $\eta$.    
\end {prop}

We won't use the explicit value of $\rho_R'$ in terms of $\rho$ here (but of course, determining the value of $\rho$ is actually one of our goals). The symmetric statement holds for the loops that touch $\eta$ from the left, changing $\rho$ into $\kappa -6- \rho$ to get the formula for $\rho_L'$. 

Similarly, \cite {cle_percolations} contains a description of the conditional law of the $\CLE_{\kappa'}$ loops when one conditions on $\eta$ up to a stopping time $\tau$. In that case, 
the loops touching $\eta$ from the left and from the right are not conditionally independent anymore (to start with, they have to remain disjoint), but the previous proposition provides a recipe to construct them: First sample the rest of $\eta$, and then apply Proposition~\ref{BCLEinCLE}. Using the imaginary geometry setup, one can then directly view the joint law of these interface-touching loops without tracing the rest of $\eta$. We will use this description at some point in our proof (the precise result from \cite{cle_percolations} will be easier to state then).

We are now ready to prove Theorem~\ref{thm3}: 
The next three sections will be devoted to the proofs of~\eqref{it:unexplored}, \eqref{it:bl_proc}--\eqref{it:jump_ratio} and \eqref{it:disks_are_disks}, respectively.

\subsection {Proof of stationarity} 

We are going to follow a similar strategy than for our proof when $p=1$.  It is this time convenient (we hope it will become immediately clear why) to start off with a weight $4$ quantum wedge $\CWW = (\h,h,0,\infty)$.  We let $\eta$ be an independent $\SLE_\kappa$ on $\h$ from $0$ to $\infty$.  We can view $\eta$ as the zero angle flow line of a GFF $h^\IG$ on $\h$ with boundary conditions given by  $\lambda = \pi / \sqrt {\kappa}$ on $\R_+$ and  $-\lambda$ on $\R_-$.
Then we know that the law of $\CWW$ is invariant under the operation of cutting along $\eta$ for a given amount of quantum length and then conformally mapping back to $\h$ (this is the $\rho_1=\rho_2=0$ case of Theorem~\ref{thm:gluing_wedges}, which is also the basic quantum zipper result from the earlier paper \cite{SHE_WELD}) -- see Figure \ref{f15} for a sketch.

\begin{figure}[ht!]
\includegraphics[width=5cm]{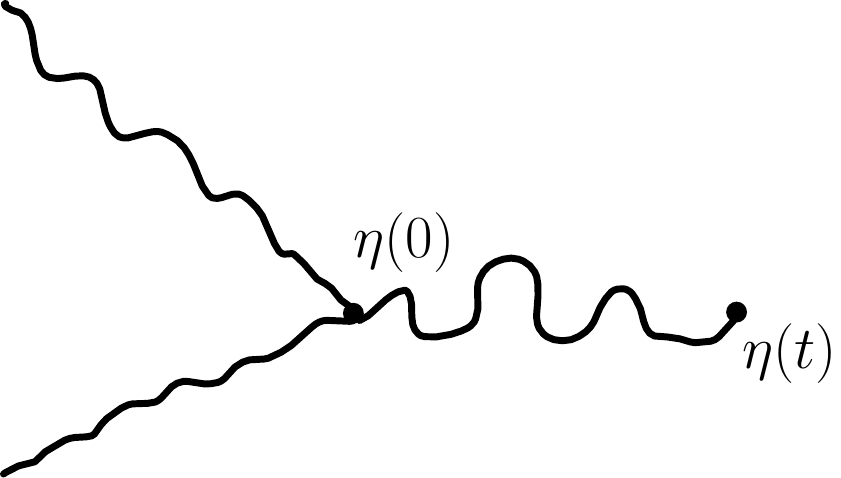}   
	\caption{Both $\CWW$ and $\CWW_t= \CWW \setminus \eta [0,t]$ (to the right of the drawn paths) are wedges of weight $4$ (it would be more natural to view $\CWW$ as the complement of a path, but for representation purposes, we draw it more like a wedge).}
	
	\label{f15}
\end{figure}

The idea is now to construct the weight $W_D = \gamma^2-2$ wedge $\CW$ (and the generalized half-plane $\CH$) within $\CWW$ in such a way that $\eta$ will be the $\SLE_\kappa (\rho; \kappa-6-\rho)$ process that slices through its beads. 

This is naturally done in the imaginary geometry framework of \cite{MS_IMAG} (we will use the notation $\chi = 2/\gamma-\gamma/2$ and $Q  = 2/\gamma+ \gamma/2$)  as follows: 
Define 
\[ \theta_0 = \frac{2\pi(\gamma^2-2)}{4-\gamma^2} = \frac{2\pi(\kappa-2)}{4-\kappa}\]
and choose $\theta \in (0, \theta_0)$ (note that $\theta_0 > 0$ since $\kappa' \in (4,8)$ so that $\gamma \in (\sqrt{2},2)$).  Let $\eta_L$ (resp.\ $\eta_R$) be the flow line of $h^\IG$ with respective angles  $\theta$ and $\theta-\theta_0$.

Since $\theta - \theta_0 < 0 < \theta$, the imaginary geometry results from \cite{MS_IMAG} show that $\eta$ is squeezed in between $\eta_R$ and $\eta_L$, and that $\eta_R$ and $\eta_L$ intersect. Moreover (see Figure \ref {f16}) conditionally on $\eta_L$ and $\eta_R$, the path $\eta$ turns out to be an $\SLE_\kappa (\rho; \kappa-6-\rho)$ process slicing through the connected components that are in between $\eta_L$ and $\eta_R$ (which is why we chose $\theta_0$ like this), where 
\[ \rho = \frac{\theta \chi}{\lambda} - 2 = \frac{\theta}{\pi} \left(2-\frac{\kappa}{2}\right) - 2.\]
We can note that when $\theta$ varies from $0$ to $\theta_0$, then $\rho$ spans through all of the values between $-2$ and $\kappa -4$.

\begin{figure}[ht!]
\includegraphics[width=12cm]{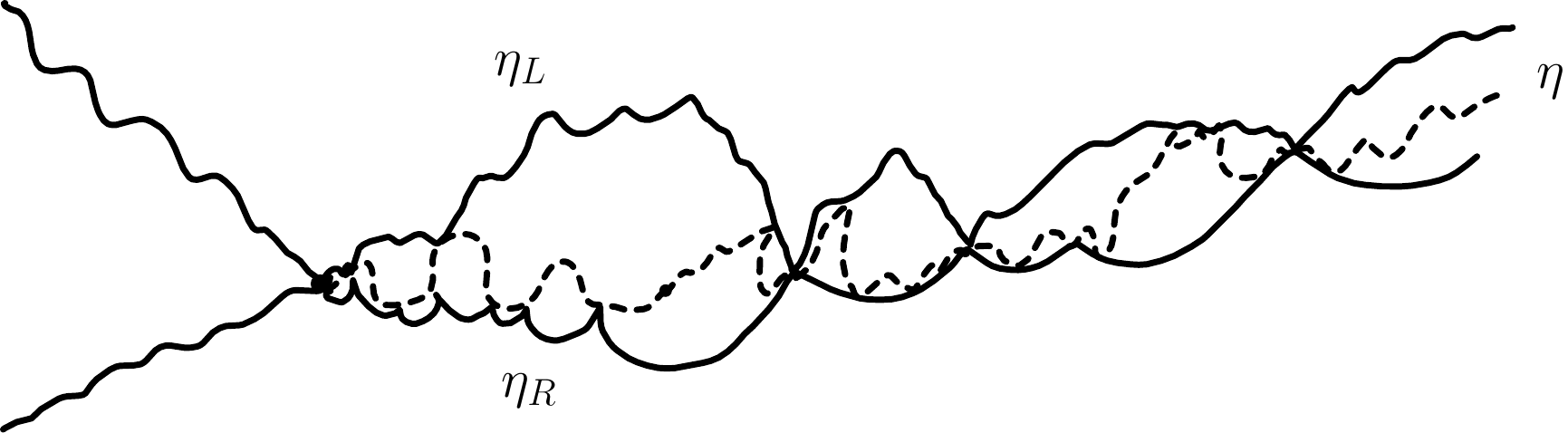}   
	\caption{The conditional law of $\eta$ given $\eta_L$ and $\eta_R$ is an $\SLE_\kappa (\rho ; \kappa-6- \rho)$ in the wedge ${\mathcal W}$ between these two curves.}
	
	\label {f16}
\end{figure}

We can furthermore consider  the counterflow line $\eta_L'$ of $h^\IG + (\theta+\pi/2) \chi$ from $\infty$ to $0$.  This is the counterflow line chosen so that  its  right boundary is $\eta_L$ (indeed, it is the flow line of $h^\IG + (\theta+\pi/2) \chi$ with angle $-\pi/2$). Similarly, the left boundary of the counterflow line $\eta_R'$ of $h^\IG + (\theta-\theta_0-\pi/2)\chi$ is equal to $\eta_R$.

\begin{figure}[ht!]
\includegraphics[width=12cm]{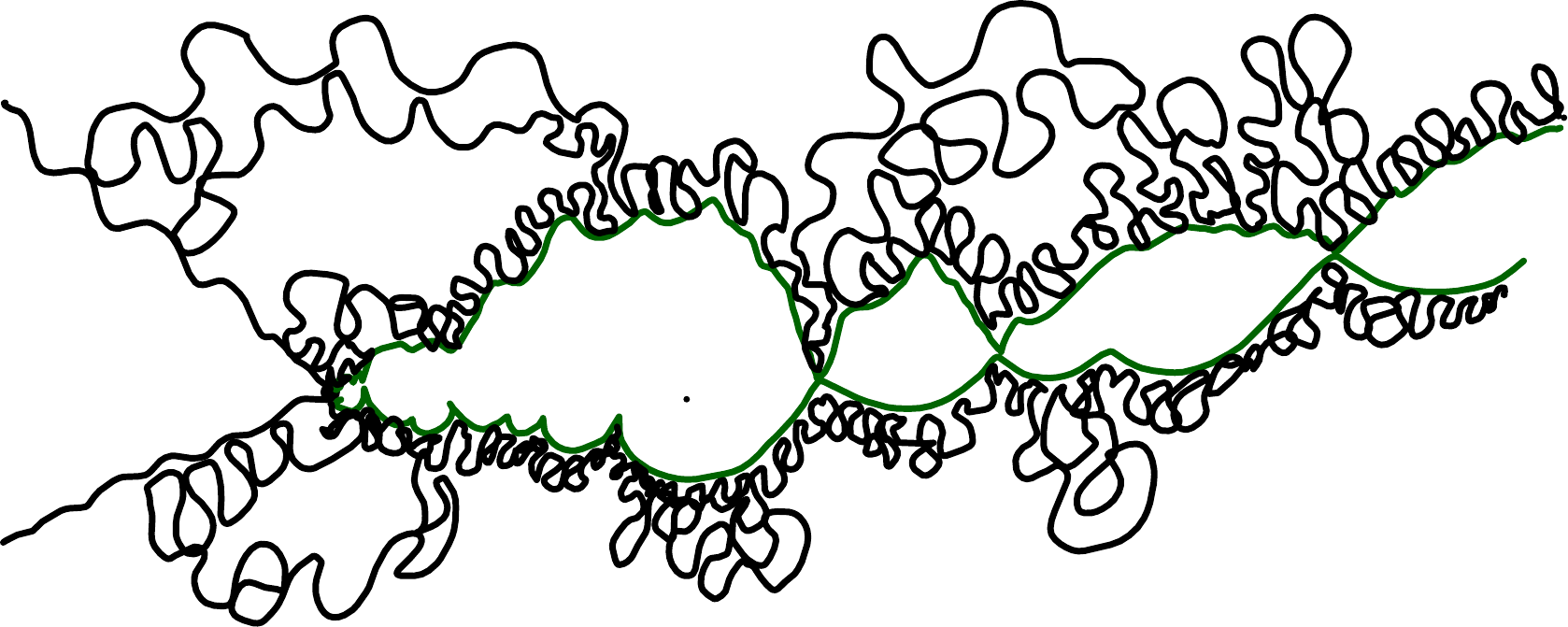}   
	\caption{The part between $\eta_L'$ and $\eta_R'$ is the generalized half-plane ${\mathcal H}$ }
	
	\label{f17}
\end{figure}

We can now combine these curves with the independent LQG structure of $\CWW$. Our choice of parameters (that ensure that the joint law of $\eta_R$ and $\eta_L$, i.e., the angle difference $\theta_0$, is the right one for this) and Theorem~\ref{thm:gluing_wedges} show that the quantum surface parameterized by the bubbles which are between $\eta_R$ and $\eta_L$ is a quantum wedge of weight 
\[ W = \frac{\theta_0}{\pi} \chi \gamma =  \gamma^2-2\]
(see \cite[Table~1.1]{dms2014mating} for this relationship between wedge weight and imaginary geometry angle). This will be our quantum wedge~$\CW$.

Furthermore, the quantum surface in-between $\eta_L'$ and $\eta_R'$ (see Theorem~\ref{thm:gluing_wedges'} and the remarks after that) is then a doubly-forested wedge of weight $W_D=\gamma^2-2$ -- this will be our generalized quantum half-plane~$\CH$, see Figure \ref {f17}.

We now parameterize $\eta$ according to its quantum length, and fix some positive $t$. Then, as we have already pointed out, 
$(\h \setminus \eta([0,t]), h, \eta(t), \infty)$ is again a quantum wedge of weight $4$. 
We let $f_t \colon \h \setminus \eta([0,t]) \to \h$ be the centered Loewner map.  Then the quantum surface $(\h, h_t, 0, \infty)$ with $h_t = h \circ f_t^{-1} + Q \log|(f_t^{-1})'|$ is a weight-$4$ quantum wedge.
On the other hand,  $\wh h_t^\IG = h^\IG \circ f_t^{-1} - \chi \arg( (f_t^{-1})')$ has the same law as $h^\IG$. 
We then define $\wh \eta_{L,t}'$ (resp.\ $\wh \eta_{R,t}'$) from $\wh h_t^\IG$ in the same way in which $\eta_L'$ and $\eta_R'$ are defined from $h^\IG$ (i.e., they are the counterflow line from $\infty$ to $0$ of $\wh h_t^\IG + (\theta+\pi/2)\chi$ and $\wh h_t^\IG + (\theta-\theta_0-\pi/2)\chi$).

In the same way as for $t=0$, the quantum surface parameterized by the components of $\h \setminus (\wh \eta_{L,t}' \cup \wh \eta_{R,t}')$ which are between the right boundary of $\wh \eta_{L,t}'$ and the left boundary of $\wh \eta_{R,t}'$ form a quantum wedge $\wh \CW(t)$ of weight $\gamma^2-2$, and if we include the forested lines corresponding to the loops of $\wh \eta_{L,t}'$ (resp.\ $\wh \eta_{R,t}'$) on the right (resp.\ left), then we get a generalized quantum half-plane $\wh \CH(t)$ as before.
We can map these surfaces forward via $f_t^{-1}$ (so the change-of-domain formula gives again the field $h$) to obtain representatives of this generalized quantum half-plane and this generalized quantum wedge -- we denote them by $\wt \CH(t)$ and $\wt \CD(t)$. 

We now want to argue that $\wt \CH(t)$ has the same law as $\CH(t)$ as defined in the theorem, i.e., by first sampling $\CW$ and $\CH$, then sampling a $\CLE_{\kappa'}$ in $\CW$, and exploring along its interface up to the quantum natural time $t$ and attaching all discovered $\CLE_{\kappa'}$ loops . Both $\CH(t)$ and $\wt \CH(t)$ are obtained by cutting out pieces from $\CH$ that are independent of $h$, so it suffices to focus on the geometry of these pieces. 
Let us now define   
$\wt{\eta}_{L,t}' = f_t^{-1}(\eta_{L,t}')$ and $\wt{\eta}_{R,t}' = f_t^{-1}(\eta_{R,t}')$.  Then $\wt{\eta}_{L,t}'$ agrees with $\eta_L'$ until it first hits the left side of $\eta([0,t])$ after which it branches toward $\eta(t)$.  Likewise, $\wt{\eta}_{R,t}'$ agrees with $\eta_R'$ until it first hits the right side of $\eta([0,t])$ after which it branches towards $\eta(t)$, see Figure \ref{f18}. 

\begin{figure}[ht!]
\includegraphics[width=12cm]{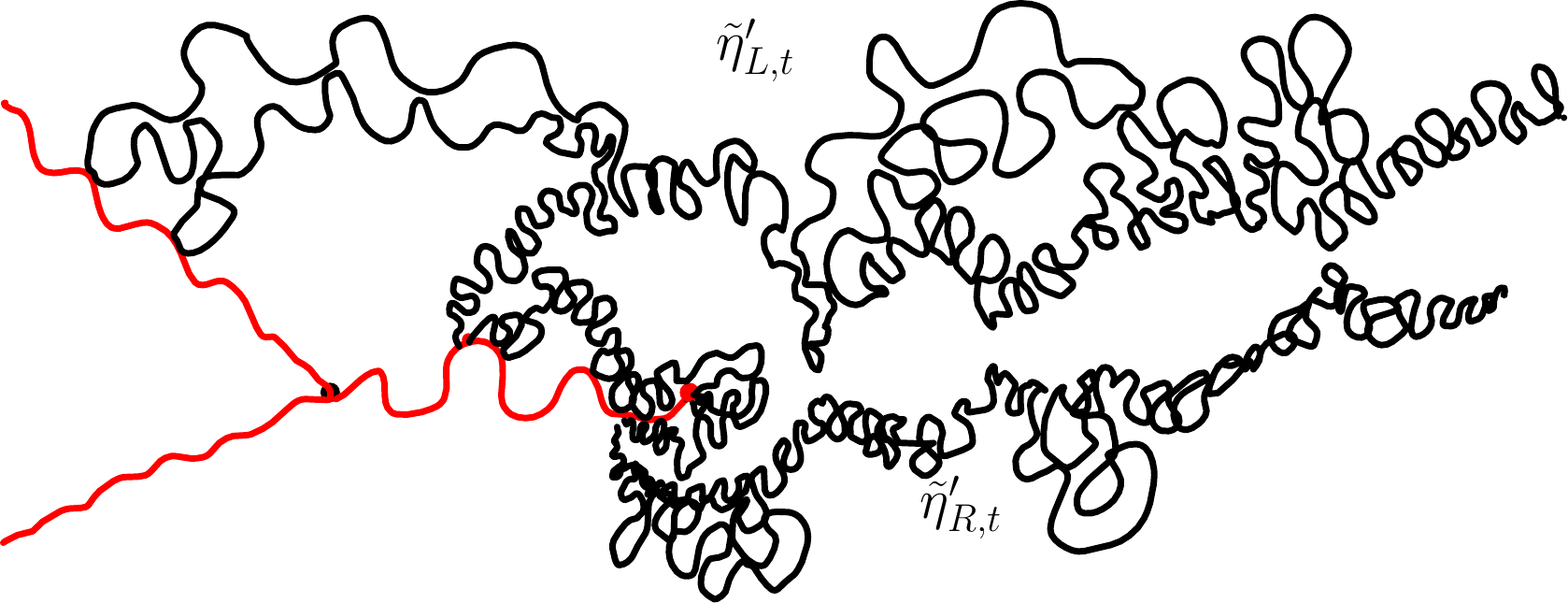}  
	\caption{The paths $\wt{\eta}_{L,t}'$ and ${\wt {\eta}_{R,t}'}$ agree with $\eta_L'$ and $\eta_R'$ up until they hit $\eta[0,t]$. They define a quantum half-plane in $\CWW_t$ in the same way in which $\eta_{L}'$ and ${\eta_R'}$ defined the generalized half-plane ${\mathcal H}$ in $\CWW$.}
	
	\label {f18}
\end{figure}

So, we can interpret the parts of these paths before these branching times as defining the boundaries of the initial generalized half-plane $\CH$ up until (and including) that of the ``currently explored disk'' $\CD(t)$ (in which $\eta(t)$ is).  

If we sample a colored $\CLE_{\kappa'}$ in this disk, we know that $\eta$ can be viewed as the interface from one of its marked points to the other. Now is the time where we will use the imaginary geometry type description from \cite {cle_percolations} of the conditional distribution of the $\CLE_{\kappa'}$ loops that touch $\eta$ up to some finite time.  Indeed, it precisely says (see \cite[Figure~9.2]{cle_percolations} and the surrounding text), that conditionally on $\eta$ up to a stopping time (and here $t$ can be viewed as a stopping time as it involves the conditionally independent field $h$), the outermost pieces of $\CLE_{\kappa'}$-loops attached to it have exactly the same joint law of  the parts of $\wt \eta_{L,t}$ and $\wt \eta_{R,t}$ after the splitting (see Figure \ref {f19}). This proves exactly that the conditional law of $\CH(t)$ is that of $\wt \CH(t)$.   
Indeed,  $\wt{\eta}_{L,t}'$ (resp.\ $\wt{\eta}_{R,t}'$) is the counterflow line of $h^\IG+c_L$ (resp.\ $h^\IG + c_R$) in $\CD(t) \setminus \eta([0,t])$ starting from the first point on $\CD(t)$ visited by $\eta$ and targeted at $\eta(t)$ where
\[ c_L =  \lambda' - \lambda + \theta \chi = (\theta-\pi/2)\chi \quad\text{and}\quad c_R = \lambda - \lambda' + (\theta-\theta_0) \chi = (\theta-\theta_0 +\pi/2)\chi\]
(The reason that $c_L = (\theta-\pi/2)\chi$ and not $(\theta+\pi/2)\chi$ is that in the above it is growing from the bottom of $\CD(t)$ rather than the top, which corresponds to a change of angle of $\pi$ in the clockwise direction.  Recall that a change of angle $\pi$ in the clockwise direction corresponds to a lowering of field heights by $\pi \chi$.  This also explains the value of $c_R$.)

\begin{figure}[ht!]
\includegraphics[width=12cm]{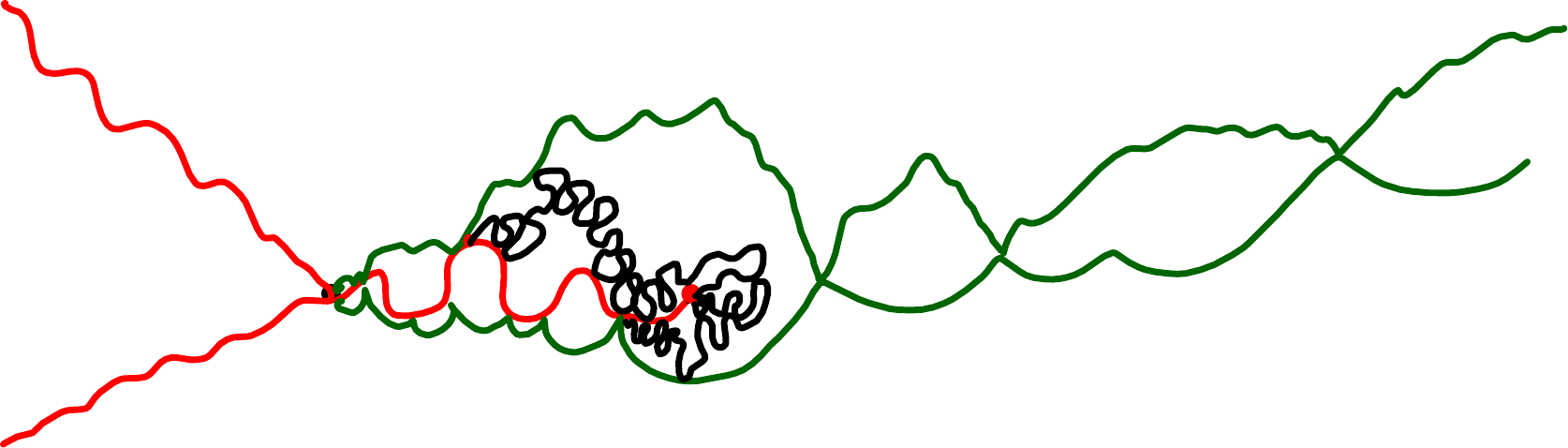}   
	\caption{The paths $\wt {\eta}_{L,t}' \setminus \eta_L'$ and $\wt {\eta}_{R,t}' \setminus \eta_R'$ can be interpreted as the contours of the colored $\CLE_{\kappa'}$ loops in $\CW$ that contribute to the boundary of $\CW(t)$.}
	
	\label {f19}
\end{figure}

Finally, to see that $\CW(t)$ is independent of $\CF_t$, one proceeds as in the case $p=1$, noting that on the one hand, $\wt \CH(t)$ and the two quantum surfaces that lie on the other ``sides'' of $\wt \CH(t)$ of $\eta_{R,t}'$ and of $\eta_{L,t}'$ are three independent surfaces, and using the $\BCLE_{\kappa'} (\rho')$ construction of \cite{cle_percolations} that shows that the cut out domains can be constructed by drawing appropriate SLE-type curves (to complete the BCLE branching trees) in those two other quantum surfaces. \qed

\subsection {L\'evy processes and their jumps}
We now prove the statements \eqref{it:bl_proc}--\eqref{it:jump_ratio} of Theorem~\ref{thm3}. 

Just as in the case $p=1$, statement (i) implies that $L$ and $R$ are both  $\alpha'$-stable L\'evy processes. Let us now explain why they are independent: In the previous setup, when one first samples all of $\eta$ in the wedge $\CWW$ of weight $4$, then it splits it into two independent wedges of weight $2$. The processes $R$ and $L$ can then be read off from what happens 
on either side of $\eta$ by drawing independent BCLE processes on each of the sides (this is one of the main results of \cite {cle_percolations}; we will come back to this description in the next section), so that they are necessarily independent as well. 

We can note that up until the interface hits a certain $\CLE_{\kappa'}$ loop, it is actually independent of the color of that loop. That loop will therefore correspond to a positive jump of $L$ or of $R$ with respective probabilities $p$ and $1-p$. This indicates that the ratio between the rate of positive jumps of $L$ and the rate of positive jumps of $R$ is $p/ (1-p)$.   

We now turn to the ratio $U_R$  between the rates of positive and negative jumps of $R$ (by symmetry, the corresponding ratio for $L$ is then obtained by changing $\rho$ into $\kappa-6- \rho$). We use the same idea as in the case $p=1$: We know that $R$ is an $\alpha'$-stable process, and it is then a standard result (see \cite[Chapter~VIII, Lemma 1]{bertoin96levy}) for such processes that if we set $R^{\#}_t := \inf_{s \leq t} R_s$, then the set ${\mathcal R} := \{ - R^{\#}_t, t \ge 0 \}$ can be viewed as the range of a stable subordinator, whose index $1/ \alpha''$ can be explicitly written in terms of $\alpha'$ and $U_R$. In particular, 
$$ U_R = \frac {\sin (\pi ( \alpha' - \alpha''))}{\sin (\pi \alpha'')}$$ 
(as above, this is a slight reformulation of  \cite[Chapter~VIII, Lemma~1]{bertoin96levy}, with $\alpha''$ given by $\alpha'$ times the positivity parameter of $-R$).

But here, we have another way to express this jump distribution via the weight of the wedges that are obtained by slicing the wedge $\CW$ by the interface $\eta$. We know that the wedge $\CW_R$ to the right of $\eta$ has weight $W=\kappa- 4 - \rho$. By Remark~\ref{scalingremark}, we get a Poisson point process of generalized boundary lengths of intensity
$$dy / y^{2 + (\rho/2) - (4/\kappa') + 1}.$$
Just as in the case $p=1$ (and this is the main observation here!), we note that the minimum of $R$ is reached only at the times at which $\eta$ touches the right-hand boundary of $\CW$, and these points correspond exactly to the point ``between'' the beads of $\CW_R$. Furthermore, the corresponding jumps will have the same scaling property as the generalized boundary lengths of these beads.
Hence, we get that 
\[\alpha'' =  2 + \frac \rho 2  - \alpha'.\]

Plugging this relation into the expression for $U_R$, we finally get 
\[  U_R = \frac{\sin(2\pi \alpha' - \pi \rho/2)}{\sin(\pi \rho/2- \pi \alpha')}.\]
\qed

In what follows, we will denote by  $A_+ = A_+(p) $ and $A_- = A_- (p)$ the rates of positive jumps and the rates of negative jumps (when compared to a standard stable subordinator) of the $\alpha'$-stable process $R + L$. We will also denote by $A_{+,R}(p)$, $A_{-,R}(p)$, $A_{+,L} (p)$ and $A_{-,L} (p)$ the rates of positive and negative jumps of $R$ and $L$ respectively. We then have (either from the definition or from Theorem \ref {thm3}\eqref{it:bl_proc}--\eqref{it:jump_ratio}) that $A_{+,L} (p) = p A_+(p)$, $A_{+,R}(p) = (1-p) A_+(p)$, 
$A_{+,R} (p) / A_{-,R} (p) = U_R$,  $A_{+,L} (p) / A_{-,L}(p) = U_L$ and $A_{-,R} (p) + A_{-,L}(p) = A_-(p)$. 

\subsection{Cut-out domains are quantum disks}
We now turn to prove part~\eqref{it:disks_are_disks} of Theorem~\ref{thm3}. 

We start with proving a general statement on BCLEs on quantum half-planes, which is the $\BCLE_{\kappa'} (\rho')$ analog of Proposition~\ref{main2}. It can also be viewed as the analog for non-simple BCLEs of the corresponding result \cite[Proposition~4.4]{CLE_LQG} (though in \cite{CLE_LQG} it is stated for a quantum disk, but the result also holds in the setting of the quantum half-plane).

Suppose that $\kappa' \in (4,8)$ and  $\rho' \in (\kappa'/2-4,\kappa'/2-2)$.  We now draw such a $\BCLE_{\kappa'} (\rho')$ on top of a quantum half-plane ${\mathcal H}$, and we then define the discovery process of the $\BCLE_{\kappa'}(\rho')$ loops in exactly the same way as the $\CLE_{\kappa'}$ discovery (in the $p=1$ case). For each $u \ge 0$, we define the quantum surface ${\mathcal H}(u)$ that is ``outside'' the union of all BCLE loops that are touching $(-\infty, x(u)]$. We then define $\wt L$ in exactly the same way in which $L$ was defined out of a $\CLE_{\kappa'}$ in Section~\ref{S3.1}. The positive jumps of $\wt L$ correspond to the discovery of a boundary-touching BCLE loop. 

\begin{prop}
\label{prop:bcle_on_half_plane} The quantum surfaces that are cut out at the positive and negative jumps of $\wt L$ are two independent Poisson point processes of generalized quantum disks.
\end{prop}

\begin{proof}
We can use the representation of the BCLE in terms of counterflow lines $\eta_x'$ of a GFF $h^\IG$ on $\h$ with boundary conditions given by $-\lambda'(1+\rho')+\pi \chi$. 
The proof of the result then follows in exactly the same manner as the proof of Proposition~\ref{main2}. Indeed, if for any $x \in \R$ we condition on the flow lines of $h^\IG$ from $x$ to $\infty$ 
with angles $\pm \pi/2$ the conditional law of $\eta'$ between the two flow lines is an $\SLE_{\kappa'}(\kappa'/2-4;\kappa'/2-4)$.   
\end{proof}

We can now deduce part~\eqref{it:disks_are_disks} of Theorem~\ref{thm3} using Proposition \ref {BCLEinCLE}. In the context of the proof of part~\eqref{it:unexplored} of Theorem~\ref{thm3} given above, the $\SLE_\kappa$ process $\eta$ that we started with drawn on top of an independent quantum wedge $\CW$ of weight $4$ divides it into independent wedges $\CW_1,\CW_2$ of weight $2$ (i.e., quantum half-planes).  The counterflow line exploration on the left and right sides of $\eta$ described in the proof of part~\eqref{it:unexplored} of Theorem~\ref{thm3} corresponds to drawing a $\BCLE_{\kappa'}(\rho_L')$ and an independent $\BCLE_{\kappa'}(\rho_R')$ respectively in $\CW_1$ and $\CW_2$ --  the actual values of $\rho_L'$ and $\rho_R'$ in terms of $\rho$  do not really matter for our purpose here. This already shows that the Poisson point processes of quantum surfaces associated to the jumps of $R$ is independent from the Poisson point processes of quantum surfaces associated to the jumps of $L$. 
The jumps of $L$ and the jumps of $R$ then correspond to the jumps of the discovery process of those $\BCLE$, and we can therefore apply Proposition~\ref{prop:bcle_on_half_plane} to conclude.
\qed

\section {Explorations of generalized quantum disks and consequences} 

\label {S5}
The proofs of the statements in this section are  almost identical to the corresponding ones in \cite{CLE_LQG} for simple CLEs on LQG. We will therefore only quickly browse through the results, referring the reader to \cite{CLE_LQG} for the proofs. 

\subsection {New jump rates and branching tree structure} 
\label {S51}
We now consider a generalized quantum disk $\CD$ with a generalized boundary-typical marked point $x_0$ (chosen uniformly according to the 
generalized boundary length measure). We sample an independent colored $\CLE_{\kappa'}$ inside each of its beads, and we start exploring the interface between red and blue loops starting from $x_0$. This interface $\eta$ is defined in the same way as for the exploration of a generalized quantum half-plane, that one again chooses to parameterize according to its quantum length.  One discovers in the same way the $\CLE_{\kappa'}$ loops that $\eta$ intersects, and defines in the same way the decreasing family of quantum surfaces $\CD(t)$, except that one has to make the following branching rule (to replace the marked target point at infinity of the generalized quantum half-plane): Whenever the trunk disconnects the remaining to be discovered domain into two pieces, the process chooses to go into the direction of the piece with largest generalized boundary length. 

In this way, one has a curve $\eta(t)$ defined up to some stopping time $T$ (corresponding to the time at which the boundary length of the remaining to be discovered domain vanishes). At time $t$, the generalized boundary length of the remaining to be discovered domain $\CD(t)$ will be denoted by $\Lambda_t$. This process will make a positive jump whenever $\eta$ discovers a blue or a red loop, and it will make a negative jump whenever it ``disconnects'' $\CD(t)$ into two pieces (in the same way as for the exploration of generalized half-planes). 

The results for such explorations can be summarized as follows.

\begin {theo} 
\label {thm4} 
The following hold for this exploration of colored $\CLE_{\kappa'}$ on an independent generalized quantum disk. 
\begin {itemize} 
 \item For each $t \ge 0$, conditionally on $(\Lambda_s)_{s \le t}$, all cut out domains, the interior of the discovered loops and the quantum surface $\CD(t)$ are independent generalized quantum disks (with lengths respectively given by the corresponding jumps of $\Lambda$ and by $\Lambda_t$).
 \item The process $(\Lambda_t)_{t \ge 0}$ is a pure jump-process with rates of jumps (with respect to Lebesgue measure) when $\Lambda_t = \ell$ given by $\nu_\ell (l) dl$ on $\R$, where for all $l \ge 0$,  
 $$ \nu_\ell(l) :=  \nu (\ell, \ell +l) := A_+(p) \frac {\ell^{\alpha'+1}}{l^{\alpha' +1} (\ell +l)^{\alpha'+1}}$$ 
 and 
 $$ \nu_\ell (-l) := \nu (\ell, \ell -l) :=  A_-(p)   \frac {\ell^{\alpha'+1} 1_{ l < \ell / 2}}{l^{\alpha' +1} (\ell -l)^{\alpha'+1}} .$$
\end {itemize}
\end {theo}

The proof, which is essentially identical to that of the corresponding statements in \cite{CLE_LQG} is based on the local absolute continuity between generalized quantum half-planes and generalized quantum disks, as well as on the scaling properties of the generalized boundary length, and of course on Theorem~\ref{thm3}.

The main intermediate step is to first consider the following setup: Suppose that one starts the exploration at a boundary-typical point $x(0)$ of a generalized quantum disk of boundary length $\wt \ell_0$, and that $y$ is the boundary point that lies at counterclockwise generalized boundary length $r_0$ from $x(0)$. We then define the interface $\eta$ in the same way, except that the branching rule when 
the trunk disconnects the remaining to be discovered domain into two pieces is now that the process chooses to go in the direction of $y$. At each time $t$ smaller than the total quantum length $T$ of $\eta$,  we can then define the generalized boundary lengths $\wt L_t$ and $\wt R_t$ of the clockwise and counterclockwise boundary arcs joining $\eta(t)$ to $y$ in the remaining to be explored generalized LQG surface. Then, it turns out that on the event $E_t = \{ t < T \}$, the law of $(\wt R_s - \wt R_0, \wt L_s - \wt L_0)_{s \le t}$ is absolutely continuous with respect to the law of $(R_s, L_s)_{s \le t}$ described in Theorem \ref {thm3}, with Radon-Nikodym derivative given by ${\wt \ell}_0^{\alpha'+1} / {\wt \ell}_t^{\alpha'+1}$ with $\wt \ell_t = \wt R_r + \wt L_t$.
We omit the details of the proof of this fact and refer to \cite {CLE_LQG}. 

One outcome is then the fact that the rates of negative jumps for $\wt R$ (from $\wt r$ to $\wt r - h$) when $\wt L= \wt l$ for this ``targeted process'' are (we write here $\wt \ell = \wt l + \wt r$) 
$$ \wt \nu ((\wt r, \wt l),(\wt r - h, \wt l)) = A_{-,R} (p) 
\frac {{\wt \ell}^{\alpha'+1} 1_{h < \wt r}}{h^{\alpha'+1} ( \wt \ell - h)^{\alpha'+1}}  .$$ 
Similarly, one has the similar formula for the rates of negative jumps of $\wt L$,
$$\wt  \nu ((\wt r, \wt l),(\wt r , \wt l - h)) = A_{-,L} (p)  
\frac {{\wt \ell}^{\alpha'+1} 1_{h < \wt l}}{h^{\alpha'+1} ( \wt \ell - h)^{\alpha'+1}} ,$$ 
and also for the rates of positive jumps -- for instance for those of $\wt R$ --  
$$ \wt \nu ((\wt r, \wt l),(\wt r + h, \wt l))= A_{+,R} (p) 
\frac {{\wt \ell}^{\alpha'+1}}{h^{\alpha'+1} ( \wt \ell + h )^{\alpha'+1}}.$$
We can note that due to the target-invariance of the exploration mechanism of the generalized disk, the processes targeting two boundary points $y$ and $y'$ will coincide up until the (negative) jump corresponding to the time at which it splits $y$ and $y'$. The rate of occurrence of negative jumps corresponding to such a splitting time should be the same for the process targeting $y$ and the one targeting $y'$; in other words, when $\wt r < \wt r'$ and $\wt r + \wt l = \wt r' + \wt l'$, one should have
$$ \wt \nu ((\wt r, \wt l),(\wt r , h )) = \wt \nu ((\wt r', \wt l'), (\wt r' -\wt r  - h,  \wt l' ))$$
for all $h < \wt r' - \wt r$.
This implies immediately that $A_{-,R}(p) = A_{-,L}(p)$. 

To then deduce Theorem \ref {thm4}, one proceeds  like in \cite {CLE_LQG}, by updating the target point after each times that are multiple of $\eps$ to be the ``antipodal'' point of the exploration in the generalized disk that remains to be explored, and to then let $\eps \to 0$. Again, we refer to \cite {CLE_LQG} for details. 

Note also that the fact that $A_{-,L} (p) = A_{-,R} (p)$ allows us to conclude the proof 
of Theorem \ref {thm:betarhoformula}: 
When $p \in (0,1)$, we get that  
$U_L / U_R = A_{+,L} (p) / A_{+,R} (p)$, 
which we know is equal to $p / (1-p)$. Plugging in our formulas for $U_R$ and $U_L$ gives indeed Theorem~\ref{thm:betarhoformula}

\begin {rema}
When one discovers a generalized quantum disk of length $1$ according to the $\CLE_{\kappa'}$ exploration procedure described in Theorem~\ref{thm3}, then the conditional expectation of the total quantum area $\CA_1$ given the information gathered at time $t$ is a martingale. Since the quantum area of the curves $\eta$ and of the $\CLE_{\kappa'}$ loops is clearly $0$, it follows that this conditional expectation is equal to 
\[ \E[ \CA_1] \times \left( \Lambda_t^{2\alpha'} + \sum_{s < t} ( \Lambda_s - \Lambda_{s-})^{2\alpha'} \right).\]
Letting $t \to 0$, one gets that the ``expected total area variation'' induced by the jumps is equal to $0$, i.e., 
$$
 \int_0^\infty \nu (1, 1 +l) [(1 + l)^{2\alpha'} + l^{2\alpha'} - 1 ] dl 
+
\int_0^{1/2} \nu (1, 1 - l) [ (1 - l)^{2\alpha'} + l^{2\alpha'} - 1 ] dl  = 0, 
$$
i.e., 
\begin {equation}
 \label {id} 
 A_+ (p) \int_0^\infty \frac {(1 + l)^{2\alpha'} + l^{2\alpha'} - 1}{l^{\alpha'+1} (1+l)^{\alpha'+1}} dl 
+
A_- (p) \int_0^{1/2} \frac {(1 - l)^{2\alpha'} + l^{2\alpha'} - 1 }{l^{\alpha'+1} (1-l)^{\alpha'+1}} dl  = 0, 
\end {equation}
We note that this identity determines uniquely the ratio $A_+(p) / A_- (p)$, which shows that this quantity is independent of $p$.  

This provides a simple way to do some sanity check on our formulas. One could for instance check from our expressions for $U_R$ and $U_L$ that indeed $A_+(p) / A_-(p) = (U_R + U_L) / 2 = -\cos (\alpha' \pi)$. 
Alternatively, we knew anyway that at $p=1/2$, $U_R = U_L = U_{R+L}$. On the other hand, the value of $U_R$ at $p=1/2$ is already known from Theorem~\ref{thm3} to be $- \cos (\alpha' \pi)$ (recall that by symmetry, in that case $\rho = (\kappa -6)/2$). We can therefore conclude  from (\ref {id}) that $U_{R+L}= - \cos ( \alpha' \pi)$ for all $p \in (0,1)$.  

It is also possible to see directly from (\ref {id}) that the ratio $A_+(p) / A_-(p)$ is
equal to $-\cos (\pi \alpha')$, which provides a further computational sanity check for our formulas for $U_R$ and $U_L$ and indicates the type of computations that are anyway behind the scenes and give rise to trigonometric functions 
(one first notes that  the integral from $0$ to $1/2$ in (\ref {id}) is one half of the integral from $0$ to $1$ of the same expression, and one can then reformulate (\ref{id}) in terms of Beta functions using some analytical continuation tricks, and one concludes using the relation with the Gamma function). 
\end {rema}

\subsection {The natural LQG measure in the CLE gasket}

Exactly as in \cite{CLE_LQG}, Theorem~\ref{thm3} can then be iteratively used in order to describe also the entire exploration tree (obtained when one also continues exploring into the cut-out disks) obtained when one explores a LQG-generalized disk on which one has drawn a nested $\CLE_{\kappa'}$ (the coloring is then not so crucial here). In particular, depending on whether one explores the entire generalized quantum disk or only the $\CLE_{\kappa'}$-gasket (the set of points in the generalized disk that are surrounded by no $\CLE_{\kappa'}$ loop), one obtains  two branching tree structures $\wt {\mathcal T}$ and ${\mathcal T}$ just as in \cite{CLE_LQG}.  Again, just as in \cite{CLE_LQG}, building on fine properties of the  usual branching processes martingales in the latter case, one can derive (in exactly the same way as in \cite{CLE_LQG}) the following fact: 

\begin {prop} 
\label {prop5}
One can define a natural LQG measure $\mu$ in the $\CLE_{\kappa'}$-gasket with the property that for a wide class of open sets $O$ (that generates the Borel $\sigma$-algebra in $D$), 
$\mu (O)$ is the limit in probability as $\eps \to 0$ of $\eps^{\alpha' + 1/2 } N_\eps (O)$, where $N_\eps (O)$ is the number of outermost $\CLE_{\kappa'}$ loops of generalized boundary length in $[\eps, 2\eps]$ that are in $O$. 
\end {prop}

This measure is also the one that appears as the natural measure on the boundary of ${\mathcal T}$ when defined by branching process martingale methods, such as in \cite{bbck}. Proposition~\ref{prop5} then shows that this measure is actually independent of the coloring of the $\CLE_{\kappa'}$ used to define the exploration tree out of the $\CLE_{\kappa'}$. Note that (as opposed to the corresponding statement in \cite{CLE_LQG}), this could have been also shown directly.  We refer to \cite{CLE_LQG} for more details and background.

\subsection*{Acknowledgements}
The authors were supported by ERC Starting Grant 804166 (JM), the NSF award DMS-1712862 (SS) and the SNF Grant 175505 (WW). WW is part of NCCR SwissMAP. We also thank anonymous referees for their comments.

\end{document}